\documentclass[aap]{imsart}

\usepackage{dsfont, pifont,amsthm,amsmath,amsfonts}
\usepackage{enumitem}
\usepackage{relsize}
\usepackage{graphicx}

\setenumerate[1]{label=\thesection.\arabic*.}
\setenumerate[2]{label*=\arabic*.}

\newcommand{\ex}{\mathbb{E}}
\newcommand{\prob}{\mathbb{P}}

\newcommand{\e}{\epsilon}
\newcommand{\ds}{\displaystyle}
\newcommand{\N}{\mathbb{N}}

\newcommand{\R}{\mathbb{R}}

\newcommand{\wt}[1]{\widetilde{#1}}
\newcommand{\lra}{\longrightarrow}
\newcommand{\mb}[1]{\mathbb{#1}}
\newcommand{\mc}[1]{\mathcal{#1}}

\newcommand{\md}{\mathrm{d}}

\newcommand{\BR}{\mathbb R}
\newcommand{\dist}{\overset{d}{=}}

\usepackage{chngcntr}

\setenumerate[1]{label=\thesection.\arabic*.}
\setenumerate[2]{label*=\arabic*.}
\counterwithin{equation}{section}

\renewcommand\thesection{\arabic{section}}

\newcounter{thm}

\theoremstyle{plain}
\newtheorem{theorem}[thm]{Theorem}
\numberwithin{thm}{section}

\newtheorem{lemma}[thm]{Lemma}
\newtheorem{prop}[thm]{Proposition}
\newtheorem{cor}[thm]{Corollary}
\newtheorem{P-C}[thm]{Theorem}
\newtheorem{H-L}[thm]{Theorem}

\theoremstyle{definition}
\newtheorem{remark}[thm]{Remark}

\begin{document}
\begin{frontmatter}

\title{Hydrodynamic limit and Propagation of Chaos for Brownian 
Particles reflecting from a Newtonian barrier}
\runtitle{Brownian Particles reflecting from a Newtonian barrier}

\thankstext{t1}{The author is currently a Zuckerman Postdoctoral Scholar at the Technion-Israel Institute of Technology. Thanks to Krzysztof Burdzy, the anonymous referee, Wai-Tong (Louis) Fan, and Spencer Gibbs for advice and suggestions. This work began at University of Washington, Seattle.}

\author{\fnms{Clayton L.} \snm{Barnes}\ead[label=e1]{cbarnes@campus.technion.ac.il}}
\affiliation{Technion-Israel Institute of Technology}
\address{Faculty of Industrial Engineering and Management,\\
Technion-Israel Institute of Technology,\\
Haifa 3200, Israel. \\ \printead{e1}}

\runauthor{Clayton L. Barnes}

\begin{abstract}
In 2001, Frank Knight constructed a stochastic process modeling the one dimensional interaction of two particles, one being Newtonian in the sense that it obeys Newton's laws of motion, and the 
other particle being Brownian. We construct a multi-particle analog, using Skorohod map estimates in proving a propagation of chaos, and
characterizing the hydrodynamic limit as the solution to a PDE with free boundary condition. This PDE resembles the Stefan problem but has a Neumann type boundary condition. Stochastic methods are used to show existence and uniqueness for this free boundary problem.
\end{abstract}

\begin{keyword}[class=MSC]
\kwd[Primary ]{60}
\kwd{60K35}
\kwd[; secondary ]{60J55}
\kwd{82}
\end{keyword}

\begin{keyword}
\kwd{Propagation of Chaos}
\kwd{Brownian motion}
\kwd{Reflected Diffusions}
\kwd{Hydrodynamic Limit}
\kwd{Skorohod Maps}
\kwd{Free Boundary Problems}
\end{keyword}

\end{frontmatter}
\section{Introduction} 

\subsection{Description of the model}
Consider $n$ Brownian particles $X_1^{(n)}(t),. .\\X_n^{(n)}(t)$ on the 
real line, reflecting from the same side of a moving barrier $Y^{(n)}(t).$ The
moving barrier is ``massive'' in the sense that it is not Brownian but 
obeys Newton's laws of motion. 
By this we mean the barrier is modeled to have momentum, and that it experiences an 
impulse upon colliding with one of the Brownian particles. Impulse is equivalent to  change in momentum, and in Newtonian physics is proportional to the change in 
velocity. In this way the Brownian particles drive the massive barrier by increasing 
(or decreasing, depending on sign) its velocity.
We assume the Brownian particles have an equal ``mass'' of $n^{-1}$ so the total mass of the system remains at unity, and we fix a constant
$K \geq 0$, the \emph{impulse constant}, which determines the strength of the 
Brownian particles' interaction with the massive barrier. Increasing $K$
will give the Brownian particles more ability to increase the massive barrier's 
velocity. If $K = 0$ the Brownian particles have no influence on the barrier,
the Brownian particles become independent reflecting Brownian motions while the barrier
will travel with constant speed. 
If $K >0$, however, the Brownian particles are dependent. This can be seen
intuitively, for in the event that one Brownian 
particle happens to impart a large impulse to the massive barrier, it
influences the barrier's trajectory and alters the region where the 
other Brownian particles are allowed to disperse themselves.
\begin{figure}
\centering
	\includegraphics[scale =.15]{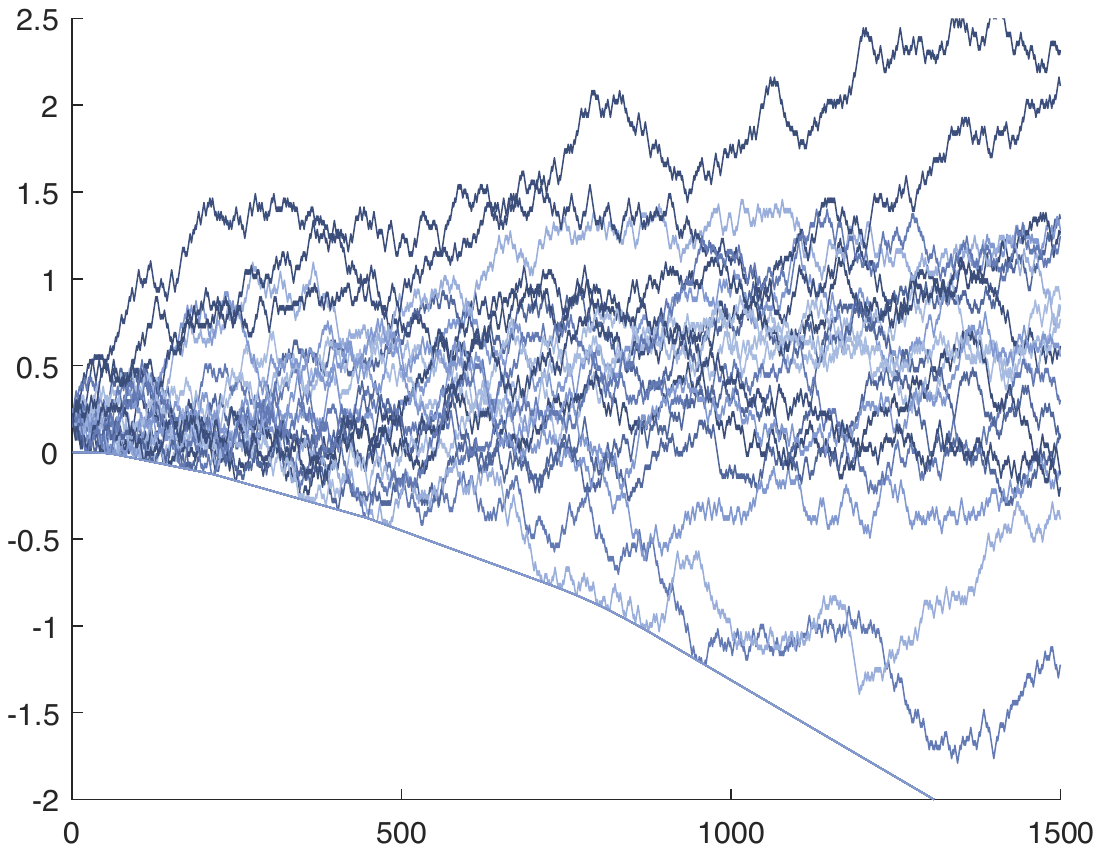}
	\includegraphics[scale = .15]{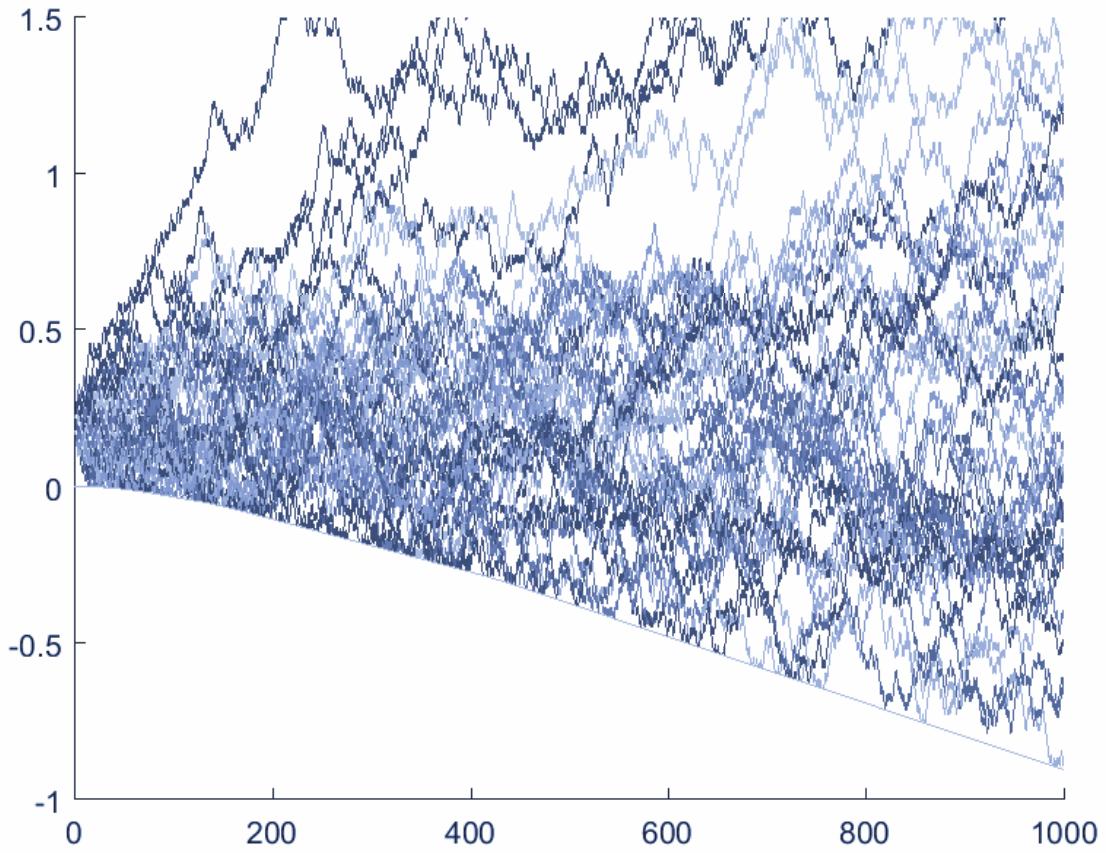}
	\includegraphics[scale =.15]{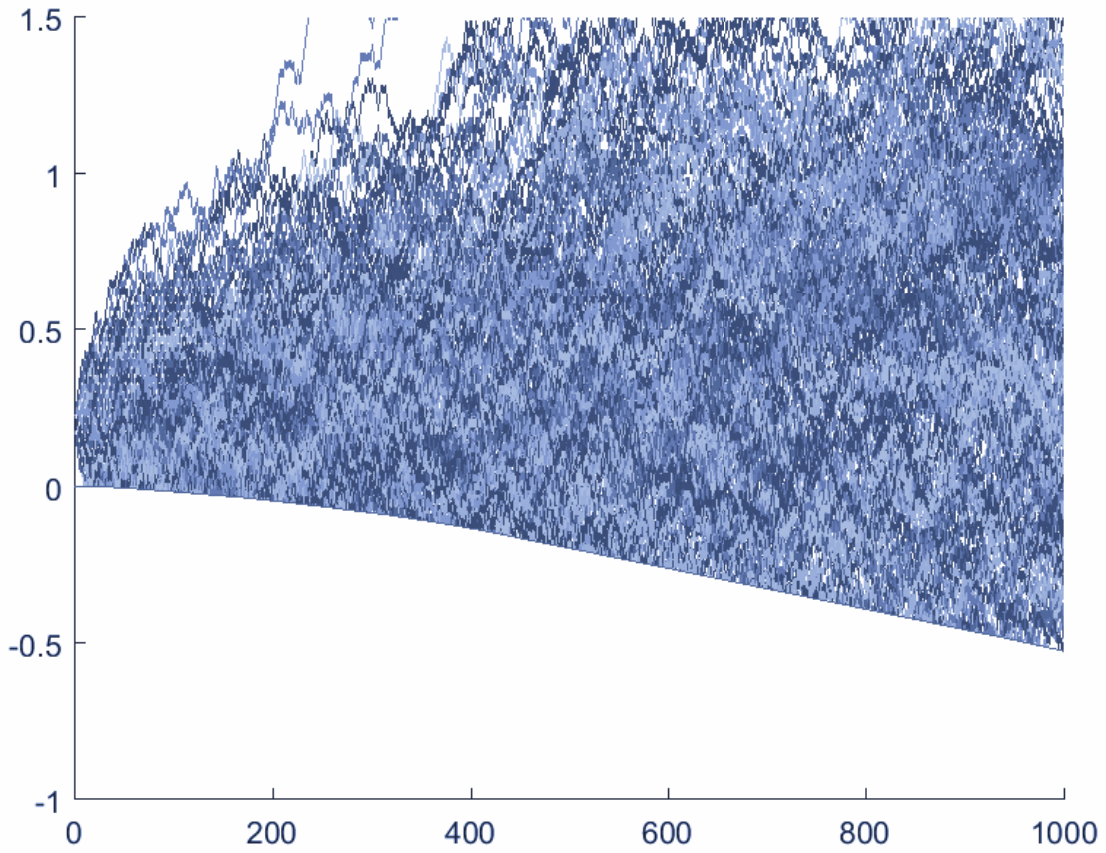}
	\caption{Simulations of 20, 40, and 200 Brownian particles reflecting from the 
	Newtonian barrier. The images were generated by the author.}
\end{figure}{} 

To formally describe the above scenario, we begin by assuming the 
standard setting:\ a filtered probability space $(\Omega,\mc{F}, (\mathcal{F}_t)_{t \ge 0}, 
\prob)$ supporting i.i.d.\ $\mathcal{F}_t-$adapted Brownian motions $B^{(1)}, \ldots, B^{(n)},$ and where $\mc{F}_t$ satisfies the usual conditions. Fix a coefficient $K \geq 0$, the impulse constant, and
an initial velocity $v \in \R$ for the massive particle. 
Consider continuous $\mc{F}_t$ adapted processes which satisfy the following system for all $t \in [0, T]$ and $i = 1, \dots, n,$ almost surely:

\begin{flalign}\label{eq:sysLaw}\begin{split}
&X^{(n)}_i(t) = B^{(i)}(t) + L^{(n)}_i(t),\\
&\frac{\text{d}\,Y^{(n)}(t)}{\text{d}t} = V^{(n)}(t) := v-\frac{K}{n}\sum_{i=1}^nL^{(n)}_i(t),\\
&X^{(n)}_i(t) \geq Y^{(n)}(t),\\
&L^{(n)}_i \text{ is nondecreasing and flat off } \{s : X^{(n)}_i(s) = Y^{(n)}(s)\}.
\end{split}
\end{flalign}
The last condition means $L^{(n)}_i$ is nondecreasing and increases only when $X^{(n)}_i$ makes contact with $Y^{(n)}.$ That is, 
\[
\int_\R1(X^{(n)}_i(s) > Y^{(n)}(s))\, \text{d}L^{(n)}_i(s) = 0.
\]
The local times are given by
\[
L^{(n)}_i(t) = \lim_{\e \to 0}\frac{1}{2\e}\int_0^t1_{[0, \e]}(X^{(n)}_i(s) - Y^{(n)}(s))\, \text{d}s \text{ for all $t \in [0, T],$}
\]
almost surely,
which is our definition of the local time of $X_i^{(n)}$ on $Y^{(n)}$. That such a system of SDEs exists with the above definition of local time is proved in Proposition \ref{prop:eqInLaw}.
We assume the initial condition $Y^{(n)}(0) = 0, \, V^{(n)}(0) = v$,
and assume $X_i^{(n)}(0)$, $i = 1, \ldots, n$ are $\mc{F}_0$ measurable
 i.i.d.\ nonnegative random variables with finite $q > 1$ moment.
We require $\mc{F}_0$ to be large enough to contain $\sigma\{X_i^{(n)}(0) : 1 
\leq i \leq n\}.$ See the figure above for sample path realizations. Existence of
a strong solution to this system is proved in Proposition \ref{prop:eqInLaw}.

\smallskip

A system $(X^{(n)}_1, \ldots, X^{(n)}_n, Y^{(n)}, V^{(n)})$ satisfying 
\eqref{eq:sysLaw} above will be called a {\it system of Brownian particles 
reflecting from a massive barrier} with {\it impulse coefficient} $K$. The 
processes $X_1^{(n)}, \ldots, X^{(n)}_n$ are the {\it Brownian particles,}
$Y^{(n)}$ is the {\it reflecting barrier} with $V^{(n)}$
its {\it velocity}. 

\subsection{Free boundary problem} 
In Theorem
\ref{HL} we characterize the hydrodynamic 
limit of the empirical process together with the random barrier.
The hydrodynamic limit is the solution to a free boundary problem given
as a pair $(p(t, x), y(t))$, both of which interact according to the PDE
below. We think of $p(t, x)$ as the temperature at time-space location
$(t, x)$ and $y(t)$ as an insulating barrier. The heat is concentrated above the insulating barrier, so
$p(t, \cdot)$ is supported on  $[y(t), \infty).$ We assume our initial 
condition $\pi_0(\md x)$ is the distribution of a random variable with finite $q$th moment for $q > 1$.

\begin{flalign}
\begin{split}\label{eq:pdeOfLimit}
&\frac{\partial \, p(t, x)}{\partial t} = \frac{1}{2}\frac{\partial^2 \, p(t, x)}{\partial x^2}, \ x > y(t),\\
&\frac{\partial^+ \, p(t, x)}{\partial x^+} = -2y'(t)p(t, x), \ x = y(t),\\
&y''(t) = -(K/2)p(s, y(s)), \, y(0) = 0, \, y'(0) = v \in \R,\\
&\lim_{t \downarrow 0}p(t, x)\md x = \pi_0(\md x),
\end{split}
\end{flalign}
\noindent
Here 
$$\frac{\partial^+p(t, x)}{\partial x^+} = \lim_{h \downarrow 0} \frac{p(t, x + h) - p(t, x)}{h}$$
 is the one sided derivative on the positive side. This second condition is mathematically equivalent to conservation of heat so the function $y(t)$ acts as an insulating barrier. Informally, one can see this by taking the time derivative of $\int_{y(t)}p(t, x)\,\md x$ and applying Leibniz integral rule. The factor of two in the second line cancels with the heat diffusion factor of one-half.
The third condition 
says the insulating barrier has an acceleration proportional to its temperature.
The last condition should be interpreted $\mc{W}_p(\pi_0, p(t, \md x)) \to 0$ as $\to 0$, where $\mc{W}_p$ is the Wasserstein/Kantorovich distance defined in subsection 1.3 and $q > p \geq 1, q \neq 2p$. See Remark \ref{remark:initialCondition}. (This odd conditions comes from Theorem \ref{W_p_Convergence}). The unique solution will be one in which the 
equalities above hold in the classical sense. That is, $p(t, x)$ is a differentiable function in its domain $\{(t, x) : 0 \leq t \leq T
, \, x \geq y(t)\},$ which is $C^1$ in time, $C^2$ in space, and $y \in 
C^2([0, T], \R).$

\subsection{Main results}
\begin{theorem}
There exists a unique classical solution to the free boundary problem 
\eqref{eq:pdeOfLimit}.
\end{theorem}

\noindent

For the hydrodynamic limit we consider the empirical measure 
$$
\pi^{(n)}_t = \frac{1}{n}\sum\limits_{i=1}^n\delta_{\{X^{(n)}_i(t)\}}.
$$ 
For fixed $t \ge 0$, $\pi^{(n)}_t$ is a random variable with values in the space  
$\mc{P}_p(\R)$. For a time horizon $T > 0,$ $\{\pi^{(n)}_t : t \in [0, T]\}$
is a process with paths in the space $C([0, T], (\mc{P}_p, \mc{W}_p))$ with metric
\[
\|\nu' - \nu''\|_{[0,T]} := \max_{t \in [0, T]}\mc{W}_p(\nu'(t), \nu''(t)).
\]
That this process indeed has a.s.\ continuous paths is proved in Lemma~\ref{lemma:contPi}. In other words, $\{\pi^{(n)}_t : t \in [0, T] \}$ is a continuous measure-valued process. As such, $\pi^{(n)}$ induces a probability measure on $C([0, T], (\mc{P}_p, \mc{W}_p)).$ The hydrodynamic limit characterizes this distribution for 
large $n$.

\begin{H-L}\label{HL} Let $q > p \geq 1, q \neq 2p.$
Assume that for some probability measure $\pi_0$ on $[0, \infty)$ with finite $q$th moment, $\pi_0^{(n)} \ \to \pi_0$ in $\mc{W}_p$.
Then,
\begin{align}
(\pi^{(n)}, Y^{(n)}) \lra (p(t, w) \mathrm{d} w, y(t)),\ \text{ as $n \to \infty$},
\label{th: HLConvDist}
\end{align}
in distribution on $C([0, T], (\mc{P}_p, \mc{W}_p) \times \R)$, where
$y \in C^2([0, T], \R),$ $p(t, x)$ is a probability density 
supported in $[y(t), \infty)$, and with $(p(t, x), y(t))$ solving \eqref{eq:pdeOfLimit}.
\end{H-L}

\noindent
The proof is in  subsection 3.2. 
\newline

The third result is the propagation of chaos, which means the dependence of any finite 
collection of tagged particles disappears as the number of particles
tends to infinity.

\begin{P-C}[Propagation of chaos] Assume for every $n$, $X^{(n)}_i(0) = \xi_i,\, i = 1, \ldots, n$, where $\xi_i, i \in \N$ are i.i.d.\ samples of a nonnegative integrable random variable. Fix positive integers $i_1, \ldots, i_k$. Then
$$
(X^{(n)}_{i_1}, \ldots, X^{(n)}_{i_k}) \lra (X^{(\infty)}_{i_1},  \ldots, X^{(\infty)}_{i_k})\ \
$$
in distribution on $C([0, T], \R^k)$, as $n \to \infty$,
and where the limit consists of independent processes $X^{(\infty)}_{i_1},  \ldots, X^{(\infty)}_{i_k}$. Furthermore, $X_{i_k}^{(\infty)}$ are independent Brownian motions reflecting from a deterministic function $y$. That is,
\[
X_{i_1}^{(\infty)}(t) = y(t) + m(t) + B_{i_1}(t) + X_{i_1}^{(\infty)}(0),
\]
where
\begin{align*} 
&y'(t) = v - K\,\ex \, m(t), \text{ for all $t \in [0, T],$ where}\\
&m(t) = \sup_{0  \leq s \leq t}\big((B^{(i_1)}(s) + X_{i_1}^{(\infty)}(0)) - y(s)\big)^-,
\end{align*}

 and $X_{i_1}^{(\infty)} \dist \pi_0(\mathrm{d}x).$
\label{th: P-C}
\end{P-C}
The $\xi_i$ are given so the processes have an initial condition
which does not depend on $n$ in the triangular array. This ensures that after $n \geq \max i_k$ 
the initial conditions for the $X^{(n)}_{i_1}, \dots, X^{(n)}_{i_k}$ are all defined and unchanging 
with $n.$ Like Theorem \ref{HL}, the proof of Theorem \ref{th: HLConvDist} is in subsection 3.2.

\begin{remark}
For fairly general exchangeable particle systems, propagation of chaos is equivalent to weak convergence of the
empirical measures to the limiting law $P(t)$ of one particle in the space $C([0, T], \mc{P})$. Here $X(t) \dist P(t)$ and $\mc{P}$ is the space of measures, on the appropriate target space, with the metric of weak convergence. See \cite{Meleard} for a good survey and proof of this in some important examples. This is weaker than the mode of convergence with we consider.
\end{remark}
\subsection{Organization of paper}
The paper is organized as follows. Subsection 1.5 contains a historical 
background for the origin of our model and on related hydrodynamic limits.

In Section 2, we construct the processes $X^{(n)}_i$ \emph{pathwise}
on any probability space supporting an infinite sequence of
i.i.d.\ Brownian motions $B^{(1)}, B^{(2)}, \dots,B^{(n)},$ and the initial random variables $X^
{(n)}_i(0)$ for all $i = 1, \dots, n, n \in \N$. We do this by constructing a functional to which we apply pathwise to the $n$ Brownian motions $B^{(1)}, \dots
, B^{(n)}.$ In Proposition~\ref{prop:eqInLaw} we show this pathwise construction
gives a system of processes satisfying \eqref{eq:sysLaw}. Such a method for
reflected processes is called a Skorohod map since Skorohod used the method
to construct a reflected Brownian motion on the positive half-line $\BR_+ := [0, \infty)$. For instance, if $B(t)$ is a standard Brownian motion and 
$m(t) = \sup_{0 < s < t}\max\{-B(s), 0\}$, then 
$B(t) + m(t)$ has the same distribution as $X$, where $dX = dB + dL$ and
$L$ is the semimartingale local time of $X$ at zero; see \cite[Section 
3.6C]{KaratzasShreve} and the original paper by Skorohod \cite{Skorohod_map}. Here $m(t)$ would be the Skorohod map which corresponds to reflected Brownian motion.

In subsection 3.1 we introduce the lemmas and propositions  used in the proofs of Theorems \ref{HL} and \ref{th: P-C}. The proofs of these two theorems are contained in subsection 3.2.  
We use the estimates derived in the second section to demonstrate 
almost sure convergence of the barrier $Y^{(n)}$ to a unique deterministic 
function $y$ in the form of a functional strong law of large numbers; see Propositions~\ref{prop:subsequentialLimit} and \ref{prop:uniquenessSubLimit}. Here we introduce properties of the measure-valued process $\pi^{(n)}$ mentioned
above. In Proposition~\ref{prop:unifEquicPi}, we prove uniform stochastic equicontinuity, which is stronger than the typical stochastic equicontinuity 
necessary for tightness of processes in some metric space. 
We conclude the paper with Section 4, where we use our stochastic representation to prove 
uniqueness of the free boundary problem described in subsection 1.2.

\subsection{Background}
Knight introduced the model \eqref{eq:sysLaw} in the case of one Brownian particle \cite{Knight2001}. He studied density of the 
final velocity of the inert particle $Y^{(1)}$. Later, White \cite{white2007} generalized Knight's 
construction and studied several related processes. This inspired a higher 
dimensional version of a reflected process whose velocity 
vector is proportional to the boundary local time, and the stationary distribution of
its position and velocity was studied by Bass, Burdzy, Chen, and Hairer \cite{bass2010stationary}.

Historically, the study of macroscopic behavior  for systems of randomly interacting particles began in 1956 by Kac \cite{Kac} and continued with McKean \cite{McKean}
in 1969. This was followed by fundamental contributions during the 1980's by 
Sznitman \cite{Sznitman2,Sznitman1}, Tanaka \cite{Tanaka}, G\"artner \cite{Gartner} and many others. The hydrodynamic limit of a system of interacting particles is sometimes referred to as the macroscopic behavior of the system or the asymptotic behavior of the empirical measures, such
a result is closely related to propagation of chaos, and the two are equivalent equivalent when 
the system of interacting particles satisfies an exchangability condition (cf. Sznitman, M\'el\'eard). For a history of hydrodynamic limits see \cite{golse2005hydrodynamic} and \cite{chen2017systems}. Systems of randomly interacting particles are probabilistic models originally motivated by 
statistical mechanics and statistical thermodynamics, particularly the theory explored 
by Maxwell, Boltzmann, and Vlasov who describe the deterministic evolution of the distribution of 
gas. A good review of interacting particle systems of the McKean-Vlasov type is found in M\'el\'eard \cite{Meleard}. For instance, such systems are given by prescribing that the particles $X_i^{(n)}$ exhibit mean-field interaction. That is, the diffusivity and drift of each particle is a function of the particles' location and the empirical (energy) profile $\pi^{(n)}(\cdot)$. For $i = 1, \dots, n$,

\small
\begin{align}\label{eq:example1}
X_i^{(n)}(t) = X_i^{(n)}(0) + \int_0^tb(X_i^{(n)}(s), \pi^{(n)}(s))\, \md s + \int_0^t\sigma(X_i^{(n)}(s), \pi^{(n)}(s))\, \md B_i(s),
\end{align}

\normalsize
where $B_1, \dots, B_N$ are independent $d$-dimensional Brownian motions and $b(\cdot, \cdot)$ $\sigma(\cdot, \cdot)$ determine the drift and diffusivity, respectively. One usually assumes the initial conditions are i.i.d., or an appropriate asymptotic condition, to allow the propagation of chaos to hold at the initial time. For a fixed time $t$, $\pi^{(n)}(t)$ is in the space of probability measures on $\R^d$, denoted by $\mc{P}_p(\R^d)$. One places a metric on $\mc{P}_p(\R^d)$, typically the metric of weak convergence.
For a finite time horizon $T \in [0, \infty)$, $\{\pi^{(n)}(t) : t \in [0, T]\}$ is a measure valued process with paths in $C([0, T], \mc{P}_p(\R^d))$. Oelschl\"ager \cite{Karl_O} characterized 
the large-scale behavior of the system in \eqref{eq:example1} assuming $b, \sigma$ are 
sufficiently regular by demonstrating that $\pi^{(n)}(t)$ converges in distribution to $P(t)$ in 
the space $C([0, T], \mc{P}_p(\R^d))$. Here $P(t)$ is the law of the process $Z(t)$ at time $t$, 
where
\begin{align}\label{Vlasov}
Z(t) = Z(0) + \int_0^tb(Z(s), P(s))\, \md s + \int_0^t\sigma(Z(s), P(s))\, \md B(s),
\end{align}
for a $d$ dimensional Brownian motion $B$, and $Z(0) \dist \lim_{N \to \infty} X_1^{(n)}(0).$ This matches intuition upon inspection of \eqref{eq:example1}. The time varying distribution $P(s)$ is the solution to the Vlasov equation, thus giving one example of the macroscopic behavior of 
randomly interacting particles converging to a deterministic time varying distribution 
\cite{Meleard}. Furthermore, the individual processes $X_1^{(n)}$ converge to a process with time 
varying law $P$.

For other results on convergence of empirical processes, see 
\cite{varadhan1992entropy}, where Varadhan uses entropy 
methods to examine a spin 
system on a lattice when the mesh goes to zero. Entropy and relative entropy methods are general methods.
However, these are not always feasible. 
For instance, see \cite{chen2017systems},  where Chen and Fan study a system of 
particles reflecting from a separating interface. 
For an introductory reading on hydrodynamic limits, see the book \cite{
kipnis1999scaling} where Kipnis and 
Landim present a self contained treatment of hydrodynamic limits via the study of the 
generalized exclusion process and the zero-range process. These processes are continuous time
and discrete in space with spacial distance decreasing to zero.
Other hydrodynamic limit results have biological motivations
in neuron modeling. 
See \cite{pakdaman2010fluid}, \cite{de2015hydrodynamic}, and \cite[Chapter 4.3]{
greenwood2016stochastic}. Hydrodynamic limits are related to the theory of partial differential equations since the empirical measure of the particles
converge to a solution of
a PDE or free boundary problem.
In \cite{chayes1996hydrodynamic} Chayes and Swindle study the one dimensional model 
of hot random walkers which are emitted by a source and which annihilate cold 
particles which remain stationary. When a Brownian scaling is introduced, the density 
of the hot particles together with the cold region converge to the solution of the 
Stefan problem. The Stefan problem is a free boundary problem modeling the melting 
of ice next to a heat source. The heat particles are killed upon reaching the ice boundary,
 i.e.\ a 
Dirichlet boundary condition is imposed at the ice barrier, while the melting of this ice barrier is proportional to the flux of heat across it. In this way the 
density of heat and the ice barrier interact, producing the free boundary effect.
The hydrodynamic limit we study in this paper
resembles that of the Stefan problem but with some distinctive features; see \eqref{eq:pdeOfLimit}. In contrast with the Stefan problem our barrier reflects the heat back into the domain rather than absorbing it, and our barrier has an acceleration proportional to its temperature as opposed to its velocity being
proportional to the heat flux.

There is a large variety of interacting particle systems giving rise to many different limiting 
behaviors. See the above mentioned works of Tanaka \cite{Tanaka}, Sznitman \cite{Sznitman1,Sznitman2}, as well as Skorohod \cite{skorohod1987stochastic}, Nadtochiy and Shkolnikov \cite{nadtochiy2017particle}, Chen and Fan \cite{chen2017systems} to mention some. For related models of interacting particles with rank dependence, see Sarantsev 
\cite{sarantsev2015triple,sarantsev2017infinite}, Karatzas, Pal and Shkolnikov \cite{karatzas2016systems}, and Cabezas et al.\ who study out-of-equilibrium behavior of particles 
interacting through their ranks \cite{cabezas2017brownian}.

We briefly bring attention to the relatively recent study of stochastic free boundary problems. 
These are essentially SPDE's with a free boundary. See \cite{kim2012stochastic} and also in \cite{keller2016stefan}, who introduce a stochastic Stefan problem.

This article is the first in which continuity properties of Skorohod maps are used to demonstrate a hydrodynamic limit; see Section 2. By applying this method with a stochastic representation (Corollary 3.11), we prove existence and uniqueness of the free boundary problem without relying on existence and uniqueness theorems from the theory of PDEs. Properties of the transition
density for Brownian motion reflected in a time varying domains is a key ingredient for a stochastic representation of the PDE with free boundary; see \cite{burdzy2004}. This is the first existence and uniqueness result for the free boundary problem we study, as it seems not to be subsumed by known results in the analysis literature; see \cite{fasano1977generalI,fasano1977generalII,fasano1977generalIII}, for 
existence and uniqueness of the Stefan problem.

\subsection*{Notation}
For ease of reference we introduce notation which will be
used throughout the paper. First, let $(E, d)$ 
be a metric space. Anytime $\R^n$ is given we assume the standard norm.
\begin{enumerate}[resume]
\item $C(E_1, E_2)$ is the space of continuous functions from 
$(E_1, d_1)$ to $(E_2, d_2)$, equipped with the uniform metric unless
otherwise stated. We abbreviate $C([0, T], \R)$ as $C[0, T]$.
\\

\item $\mc{P}(E)$ is the space of probability measures on $E$. We may 
abbreviate $\mc{P}(\R)$ as $\mc{P}.$
\\

\item For $f \in C[0, T]$ and $[a, b] \subset [0, T]$
\[
\|f\|_{[a, b]} := \max_{x \in [a, b]}|f(x)|.
\]

\item For $f = (f_1, \dots, f_n) \in C([0, T], \R^n)$ and $[a, b] \subset [0 ,T]$
\[
\|f\|_{[a, b]} := \sum_{i = 1}^n \|f_i\|_{[a, b]}.
\]

\item For $p \geq 1,$ we denote $\mc{P}_p(E)$ as the space of probability measures on $E
$ with finite $p$th moments, and let $\mc{P}_p = \mc{P}_p(\R).$ We write $(\mc{P}_p(E), \mc{W}_p)$ for the space together
with the Wasserstein-$p$ distance
\[
\mc{W}_p(\mu, \nu) := \Big(\inf_{\substack{(X, Y)}} \ex \, d(X, Y)^p\Big)^{1/p}
\]
where the infimum is taken over random variables $X$ and $Y$ coupled on the same probability
space, such that $X \overset{d}{=} \mu$
and $Y \overset{d}{=} \nu.$ If $(E, d)$ is complete then so is $(\mc{P}_p(E), \mc{W}_p).$
We consider $p \geq 1.$ See \cite{villani2003topics}.\label{def:Wp} 
\\

\item For $f \in C([0, T], (E, d))$ and $\delta > 0$ we define the modulus of
continuity for $f$ by
\[
\omega_{(E,d)}(f, \delta) := \sup_{\substack{0 \leq s < t \leq T \\ |t - s| < \delta}}d(f(t), f(s)).
\]
\item When $\nu_t \in C([0, T], (\mc{P}_p, \mc{W}_p))$ we let $\omega'(\nu, \delta) := \omega_{(\mc{P}_p, \mc{W}_p)}(\nu, \delta).$
\\
\item $a^+ = \max\{a, 0\}$ and $a^- = \max\{-a, 0\}$ are, respectively, the positive and negative 
part of $a.$ For a function $f$ we denote this as $(f(x))^{\pm}.$
\end{enumerate}

\section{Skorohod Map: Construction and Estimates}

\noindent
In this section we construct the system given in \eqref{eq:sysLaw} by 
applying a Skorohod map to the collection of Brownian paths. 
First, we recall the classical Skorohod equation from \cite[Chapter 3.6]{KaratzasShreve}.
\begin{lemma} Let $f \in C([0, T], \R)$
with $f(0) \geq 0.$ There is a unique continuous nondecreasing function
$m_f(t)$ such that
\begin{align*}
&x_f(t) = f(t) + m_f(t) \geq 0,\\
&m_f(0) = 0, \, m_f(t) \text{ is flat off } \{s : x_f(s) = 0\}.
\end{align*}
In particular,
\[
m_f(t) = \sup_{0 \leq s \leq t} (f(s))^-.
\]
\label{classicSLemma}
\end{lemma}

\begin{remark}
The solution of the Skorohod equation has a time shift property: For any $0 \leq s \leq t \leq T,$ 
\[
x_f(t)= x_{g}(t - s),
\]
where $g(t) = x_f(s) + f(t) - f(s).$ In other words, if $\tau_x: C[0, T] \to C[0, T - x]$ is the shift
operator defined by $h\circ \tau_x(t) = h(t + x)$, then for any $s \in [0, T]$
\[
x_f \circ \tau_s = x_g,
\]
where $g = f \circ \tau_s + x_f(s) - f(s).$
\label{re:timeShift}
\end{remark}

The following lemmas will be useful later when proving tightness of our processes; see Lemma \ref{lemma:tightness}.
\begin{lemma}
Let $f, g \in C([0, T], \R)$ and assume that $f \geq g.$ Then
\[
m_f(t) \leq m_g(t), \text{ for all } t \in [0, T].
\]
\label{lemma:SkLemma1}
\end{lemma}
\begin{proof}
From Lemma \ref{classicSLemma}, 
\begin{align*}
m_f(t) = \sup_{0 \leq u \leq t}\big(f(u)\big)^- \leq \sup_{0 \leq u \leq t}\big(g(u)\big)^- = m_g(t).
\end{align*}
\end{proof}
\begin{lemma}
Let $f, y_1, y_2 \in C([0, T], \R)$ and assume that $y_1(0) = y_2(0),$  $f(0) + y_1(0) \geq 0,$ and
\begin{align}
y_1(t) - y_1(s) \geq y_2(t) - y_2(s) \text{ for all } 0 \leq s < t \leq T.
\label{eq:lemmaCond1}
\end{align}
Then
\[
m_{f+ y_2}(t) - m_{f+ y_2}(s) \geq m_{f + y_1}(t) - m_{f + y_1}(s), \text{ for all } 0 \leq s < t \leq T,
\]
where $m_{f + y_i}, i = 1,2$ correspond to the solution of the Skorohod problem provided by
Lemma \ref{classicSLemma}.
\label{lemma:skorohodIneq}
\end{lemma}
\begin{proof}
We first show that $x_{f + y_1}(t) \geq x_{f + y_2}(t)$ for all $t \in [0, T].$ That this holds
for $t = 0$ is guaranteed by the assumption on the initial conditions, which implies $x_{f + y_1}(0) =
x_{f + y_2}(0).$ 
Assume the converse, that there is some $t^* \in [0, T]$ such that
$x_{f + y_2}(t^*) > x_{f + y_1}(t^*) \geq 0.$
Let 
\[\tau = \sup\{ t < t^* : x_{f + y_2}(t) = 0\}
\] be the last zero of $x_{f + y_2}$ before time $t^*.$ Continuity of $x_{f + y_2}$ implies
$\tau < t^*.$
It follows by definition that $m_{f + y_2}$ is flat on the interval $[\tau, t^*].$ In other words,
\begin{align}
0 = m_{f + y_2}(t^*) - m_{f + y_2}(\tau) \leq m_{f + y_1}(t^*) - m_{f + y_1}(\tau).
\label{eq:sk1}
\end{align}
By shifting the Skorohod solution by time $\tau$ as in Remark \ref{re:timeShift}, using
\eqref{eq:sk1}, the fact that $x_{f + y_1}(\tau) \geq
0 = x_{f + y_2}(\tau)$, and assumption \eqref{eq:lemmaCond1},
\begin{align*}
x_{f + y_1}(t^*) &= x_{f + y_1}(\tau) + f(t^*) - f(\tau) + y_1(t^*) - y_1(\tau) + m_{f + y_1}(t^*)
- m_{f + y_1}(\tau)\\
&\geq x_{f + y_2}(\tau) + f(t^*) - f(\tau) + y_2(t^*) - y_2(\tau) + m_{f + y_2}(t^*) - m_{f + y_2}(\tau)\\
&= x_{f + y_2}(t^*) 
\end{align*}
which contradicts the definition of $t^*.$ Therefore $x_{f + y_1}(t) \geq x_{f + y_2}(t)$ for all
$t \in [0, T].$

For a fixed $s \in [0, T]$ let \[
g_i(t) = x_{f + y_i}(s) + f(t) - f(s) + y_i(t) - y_i(s) \text{ for } s \leq t \leq T,
\] and $i = 1, 2.$ The assumption \eqref{eq:lemmaCond1} on $y_i$ together with the fact that $x_{f +y_1}
 \geq x_{f + y_2}$ imply $g_1(t) \geq g_2(t).$ Apply Lemma \ref{lemma:SkLemma1} to $g_1, g_2$ and 
 shift time by $s$ as in Remark \ref{re:timeShift} to see
\begin{align*}
m_{f + y_1}(t) - m_{f + y_1}(s) = m_{g_1}(t - s) \leq m_{g_2}(t - s) = m_{f + y_2}(t) - m_{f + y_2}(s),
\end{align*}
proving the result.
\end{proof}


\begin{theorem}\label{theorem:existence_unique_Smap}
Corresponding to each $f = (f_1, \cdots, f_n) \in C([0, T], \R^n)$, with $f_i(0) \geq 0$, and $v \in \R, K \geq 0$ is a pair of continuous functions 
\[
(I^{(n)}_f(t), V^{(n)}_f(t)) =: \Gamma_nf(t) \in C([0, T], \R^{2})
\]
satisfying
	\begin{align}
		&x_i(t) := f_i(t) + I^{(n)}_f(t) + m_i(t) \geq 0,\label{eq:sMap1}\\
		&m_i(t) \text{ is flat off } \{t : x_i(t) = 0\},\label{eq:sMap2}\\
		&V^{(n)}_f(t) = -v + \frac{K}{n}\sum_{i=1}^nm_i(t), \ v \in \R,\label{eq:sMap3}\\
		&I^{(n)}_f(t) = \int_0^tV^{(n)}_f(s)\, \mathrm{d}s,
		\label{eq:sMap4}
	\end{align}
for all $t \in [0, T].$
\label{prop:existence}
\end{theorem}
\begin{remark}
It follows from the classical Skorohod equation that
\[
\ds m_i(t) = \sup_{0 \leq s < t}\big(f_i(s) + I^{(n)}(s)\big)^-.
\]
This is used in the proof of Proposition \ref{prop:eqInLaw} below.
\label{remark: sMap}
\end{remark}
\begin{proof}
\emph{Uniqueness:}
We prove a continuity estimate which holds for any solutions of \eqref{eq:sMap1} - \eqref{eq:sMap4}. Assume that \eqref{eq:sMap1} - \eqref{eq:sMap4} hold for two functions
$f = (f_1, \dots, f_n), g = (g_1, \dots, g_n) \in C([0, T], \R^n)$ and let $(I_f^{(n)}, V_f^{(n)}), (I_g^{(n)}, V_g^{(n)})$ denote the pairs corresponding to \eqref{eq:sMap3} and \eqref{eq:sMap4} for $f$ and $g$, respectively. We are assuming such solutions exist for $f, g.$
By Remark \ref{remark: sMap}, $m_i^f(t)$ is the running minimum of $f_i + I^{(n)}_f$ below zero 
until time $t$, and the same holds for $m_i^g(t)$. Hence 
\begin{align}
\|m_i^f - m_i^g\|_{[0, t]} \leq \|(f_i + I_f^{(n)}) - (g_i + I_g^{(n)})\|_{[0, t]}.
\label{eq:boundMax}
\end{align}
By the triangle inequality, \eqref{eq:sMap3}, \eqref{eq:sMap4}, and \eqref{eq:boundMax}
\begin{align}
\begin{split}
\alpha(t) &:= \sum_{i=1}^n\|(f_i + I_f^{(n)}) - (g_i + I_g^{(n)})\|_{[0, t]}\\
&\leq \, \sum_{i=1}^n\Big(\|f_i - g_i\|_{[0, t]}\Big) + n\|I_f^{(n)} - I_g^{(n)}\|_{[0, t]}\\
&\leq \|f - g\|_{[0, t]} \, + \, K\int_0^t\sum_{i=1}^n|m_i^f(s) - m_i^g(s)|\, \md s\\
&\leq \|f - g\|_{[0, t]} \, + K\int_0^t\sum_{i=1}^n\|m_i^f - m_i^g\|_{[0, t]}\\
&\leq \|f- g\|_{[0, t]} \, + K\int_0^t\alpha(s)\, \mathrm{d}s.
\end{split}
\end{align}
Now apply Gr$\text{\"{o}}$nwall's inequality to obtain
\begin{align*}
\alpha(t) \leq \|f - g\|_{[0, t]}\exp( Kt).
\end{align*}
Consequently, 
\begin{align}
\begin{split}
&\|V_f^{(n)} - V_g^{(n)}\|_{[0, t]} \leq \, \frac{K}{n}\sum_{i=1}^n|m_i^f(t) - m_i^g(t)| \leq \, \frac{K}{n}\alpha(t) \\&\leq \, \frac{K\|f - g\|_{[0, t]}}{n}\exp( Kt).
\label{eq:ineqVelocity}
\end{split}
\end{align}
This holds for any $f, g$ and any two pairs $(I_f^{(n)}, V_f^{(n)}), (I_g^{(n)}, 
V_g^{(n)})$ solving the equations \eqref{eq:sMap1} - \eqref{eq:sMap4}. Taking $g = f$ in
\eqref{eq:ineqVelocity} shows $(I_f^{(n)}, V^{(n)}_f)$ is unique, and $\Gamma_n$ is well defined assuming solutions to \eqref{eq:sMap1} - \eqref{eq:sMap4} exist.
\\
\\
\begin{remark}
The case $n=1$ is in \cite{white2007}.
\end{remark}
\noindent
\emph{Existence:} 
To demonstrate existence, we approximate with the processes 
$I^{(n)}_f, V^{(n)}_f$
that define the map $\Gamma_n$. Informally, 
we break the interval $[0, T]$ into small intervals of size $\e$ and construct $I_f^{(n, \e)}$ by
updating its velocity $V_f^{(n,\e)}$ every $\e$ step. We do this by letting
 the average minimum of $I_f^{(n, \e)} + f_i,$ $i = 1, \dots, n,$ accumulate between the steps of size
 $\e$ and adding an appropriate proportion (depending on $K, n$) of this accumulated amount to the velocity at the end of each step.
For a fixed $\e > 0,$ define the functions $I^\e_{M\e}, V^\e_{M\e}$ recursively 
in the intervals $[0, \e], [\e, 2\e], \dots, [(M-1)\e, M\e]$ as follows. 

\begin{enumerate}
	\item On the interval $[0, \e],$ let $I^\e_{\e}(t) = vt$ and  $V^\e_{\e} = v.$ 
	\item Assume we are given $I^\e_{M\e}, V^\e_{M\e}.$
	Let
	\[
	\left.I^\e_{(M+1)\e}\right|_{[0, M\e)} = I^\e_{M\e} \text{ and } 
	\left.V^{\e}_{(M+1)\e}\right|_{[0, M\e)} = V^\e_{M\e}.
	\]
	For $t \in [M\e, (M+1)\e)$ let 
	\[
	V^\e_{(M+1)\e}(t) = \frac{K}{n}\sum_{i=1}^n\max_{0 \leq u \leq M\e}\big(f_i(u) + I^\e_{M\e}(u)\big)^-
	\]
	be the average of the running minimum below zero of $f_i + I^\e_{M\e}$ until time $M\e.$ Notice that $V^\e_{(M+ 1)\e}$ is piecewise constant on subintervals
	of $[0, T]$ of the form $[j\e, (j + 1)\e), j \in \N.$
	\item Extend
	$I^\e_{(M+1)\e}$ to $[M\e, (M+1)\e)$ linearly by giving it slope $V^\e_{(M+1)\e}.$

	\item Set $I^{(n, \e)}_f, V^{(n, \e)}_f$ as the functions produced once the recursion
	covers the interval $[0, T].$ This occurs when $M$ reaches $\lceil T/\e \rceil.$
\end{enumerate}
A couple observations follow easily from this construction. First,
\[
I^{(n, \e)}(t) = \int_0^t V^{(n, \e)}(s)\, \mathrm{d}s.
\]
Second, $V^{(n, \e)}$ is monotonically increasing, and $I^{(n, \e)}$ is differentiable and convex. By construction
\[
\|V^{(n ,\e)}\|_{[0, T]} \leq |v|T + \frac{K}{n}\sum_{i=1}^n\max_{0 \leq u \leq T}\big(f_i(u)\big)^- < \infty,
\]
for every $\e > 0,$ and therefore $\{\|V^{(n, \e)}\|_{[0, T]} : \e > 0\}$ is a bounded set.
Consequently the collection $\{I^{(n, \e)} : \e > 0\}$ is uniformly Lipschitz,
and since $I^{(n, \e)}(0) = 0$ for all $\e > 0$ it is pointwise bounded as well.
Hence the family $\{I^{(n, \e)} : \e > 0\}$ satisfies the Arzel$\grave{\text{a}}$-Ascoli criterion. 
By taking a subsequence $\e_k \to 0$
there is a continuous function $I^{(n)}$ such that
\[
\int_0^tV^{(n, \e_k)}(s)\, \mathrm{d}s =: I^{(n, \e_k)}(t) \lra I^{(n)}(t)
\]
uniformly for $t$ in $[0, T]$. By the construction of $V^{(n, \e_k)},$ this implies
\begin{align*}
V^{(n, \e_k)}(t) &= \frac{K}{n}\sum_{i=1}^n\max_{0 \leq u \leq \lfloor t/\e_k \rfloor \e_k}
\big(f_i(u) + I^{(n, \e_k)}(u)\big)^-\\
&\lra \frac{K}{n}\sum_{i=1}^n\max_{0 \leq u \leq t}\big(f_i(u) + I^{(n)}(u)\big)^-
\end{align*}
uniformly for $t$ in $[0, T]$, as $\e_k \to 0.$ Set 
\[
m_i(t) = \max_{0 \leq u \leq t}\big(f_i(u) + I^{(n)}(u)\big)^-,
\] so that
\[
V^{(n)}(t) = v + \frac{K}{n}\sum_{i=1}^nm_i(t).
\]
We know $m_i$ is flat off $\{ s : f_i(s) + I^{(n)}(s) + m_i(s) = 0\}$ by Lemma \ref{classicSLemma}. 
By the dominated convergence theorem,
\[
I^{(n)}(t) = \int_0^tV^{(n)}(s)\, \mathrm{d}s,
\]
and clearly $f_i(t) + I^{(n)} + m_i(t) \geq 0.$
Therefore $(I^{(n)}, V^{(n)})$ satisfy the equations \eqref{eq:sMap1}--\eqref{eq:sMap4}.
\end{proof}
\noindent
Note that \eqref{eq:ineqVelocity} implies the map $g \mapsto V_g^{(n)}$
is Lipschitz as a map between function spaces $C([0, T], \R^n) \to C([0, T], \R)$ with 
Lipschitz constant $(K/n)\exp(KT).$ Recall

\begin{prop}(Lipschitz property of $V^{(n)}$)
For any
$v \in \R, K \geq 0$, take $f, g \in C([0, T], \R^n)$. We have
\begin{align}
\|V^{(n)}_f - V^{(n)}_g\|_{[0, T]} \leq (K\| f - g\|_{[0, T]}/n)\exp({KT)}, \label{eq:bound1}
\end{align}
and consequently
\begin{align}
\|I^{(n)}_f - I^{(n)}_g\|_{[0, T]} \leq (K\| f - g\|_{[0, T]}/n)T\exp({KT)}. \label{eq:bound2}
\end{align}
\label{prop:contV}
\end{prop}
\begin{proof}
The first bound \eqref{eq:bound1} was shown as \eqref{eq:ineqVelocity}. Notice \eqref{eq:bound1} implies \eqref{eq:bound2}, since
\begin{align*}
\|I^{(n)}_f - I^{(n)}_g\|_{[0, T]} &= \sup_{0 \leq u \leq T}|\int_0^uV^{(n)}_f(s) - V^{(n)}_g(s)\, \mathrm{d}s|\\
&\leq \sup_{0 \leq u \leq T}\int_0^u|V^{(n)}_f(s) - V^{(n)}_f(s)|\, \mathrm{d}s\\
&\leq \int_0^T\|V^{(n)}_f - V^{(n)}_g\|_{[0, T]}\, \mathrm{d}s\\
&\leq T(K\| f - g\|_{[0, T]}/n)\exp(KT).
\end{align*}
\end{proof}

\noindent
Consider the sequence $I^{(n,\e)}_f$ for a given $n$ and $f = (f_
1,\dots, f_n) \in \R^n$ defined by 2.1-2.4 in the proof of Theorem \ref{theorem:existence_unique_Smap}. By Proposition \eqref{prop:contV}, $I^{(n, \e)}_f$ 
converges in the 
uniform norm on $C([0, T], \R)$ to a unique continuous function as $\e \to 0$. The 
Proposition below says this rate of convergence only depends on $K, T,$ and $\|f\|_{[0, T]}.$

\begin{prop}
Consider the sequence $I^{(n,\e)}_f$ defined in the proof of Theorem \ref{theorem:existence_unique_Smap}. Let $f = (f_1,\dots, f_n) \in C([0, T], \R^n)$. If $l < m$, then 
$$
\|I^{(n, 2^{-l})}_f - I^{(n, 2^{-m})}_f\|_{[0, T]}  \leq ((2+K)\|f\|_{[0, T]}/n)2^{-l}\exp(KT).
$$
\label{prop:cauchy}
\end{prop}

\begin{proof}
The proof is in a similar vein as that of Proposition \ref{prop:contV}. We show bounds
between $I^{(n, 2^{-j})}_f$ and $I^j_{k2^{-j}}$ by recursively stepping through the intervals
where we update the velocity, as described informally in the beginning of the existence section. 
With these recursive bounds we attain an exponential bound over the entire interval $[0, T].$ We make some abbreviations in our notation. For $j = l, m$ we will write $I^j$ in place of
$I^{(n, 2^{-j})}_f,$ and $I^j_{k2^{-j}}$ in place of $I^{(n, 2^{-j})}_{f, k2^{-l}}.$
Recall $I^{l}$ is piecewise linear by definition. For fixed $l < m$, define
\begin{align*}
D(k) &:= \sup_{0 \leq t \leq k2^{-l}}|I^l(t) - I^m(t)| = \frac{1}{n}\sum_{i=1}^n\sup_{0 \leq t \leq k2^
{-l}}|f_i(t) + I^l(t) - (f_i(t) + I^m(t))|\\ 
&= \frac{1}{n}\|(f + I^l) - (f + I^m)\|_{[0, k2^{-l}]}.
\end{align*}
We develop bounds for $D(k)$ using a recursive argument. By construction $I^l \equiv 0$ on $[0, 2^{-l}].$ Due to nonegativity of
$I^m$, for any $t \in [0, T]$
\begin{align*}
|V^{(n, 2^{-m})}(t)| &\leq \frac{K}{n}\sum_{i=1}^n\sup_{0 \leq u \leq T}\big(f_i(u) + I^m\big)^-\\
&\leq \frac{K}{n}\sum_{i=1}^n\sup_{0 \leq u \leq T}\big(f_i(u)\big)^-\\
&\leq \frac{K}{n}\|f\|_{[0, T]}.
\end{align*}
Therefore 
$I^m$ is piecewise linear with a slope not exceeding 
$K\|f\|/n,$ and so
\begin{align}
D(1) \leq (K\|f\|_{[0, T]}/n)2^{-l}.
\label{eq:D(1)InitialBound}
\end{align}
Assume we are given $D(k)$. We wish to bound the difference
between $I^l$ and $I^m$ on the interval $[0, (k+1)2^{-l}].$ We know
$\left. I^l\right|_{[0, k2^{-l}]}= I^l_{k2^{-l}}$ and $\left. I^m\right|_{[0, k 2^{-l}]} = I^m_{k2^{-l}}$.
Similarly the function $I^l$ has constant slope on $[k2^{-l}, (k+1)2^{-l}),$
with its slope adjustment being at the end of this interval at $(k+1)2^{-l}
;$ so $I^l_{k2^{-l}} = I^l_{(k+1)2^{-l}}$ on $[k2^{-l}, (k+1)2^{-l}].$
On the other hand, $I^m$ has a slope adjustment at each time $k2^{-l} + 2^{-m}, k2^{-l} +
2^{-m + 1},  \dots, (k+1)2^{-l}.$ Note that the difference in the slope between
$I^l, I^m$ at time $k2^{-l}$ is not more than $D(k),$ so that
\begin{align*}
\|I^l_{(k + 1)2^{-l}} - I^m_{k2^{-l}}\|_{[0, (k+1)2^{-l}]} &\leq \|I^l_{(k + 1)2^{-l}} - I^m_{k2^{-l}}\|_{[0, k2^{-l}]} + KD(k)2^{-l}\\
&= \|I^l_{k2^{-l}} - I^m_{k2^{-l}}\|_{[0, k2^{-l}]} + KD(k)2^{-l}\\
&= D(k) + KD(k)2^{-l}.
\end{align*}
By this and the triangle inequality, 
\begin{align}
\begin{split}
\label{eq:cauchyBound}
&D(k+1) = \|I^l_{(k + 1)2^{-l}} - I^m_{(k+1)2^{-l}}\|_{[0, (k+1)2^{-l}]}\\
&\leq \|I^l_{(k + 1)2^{-l}} - I^m_{k2^{-l}}\|_{[0, (k+1)2^{-l}]} \ + \ \|I^m_{k2^{-l}} - I^m_{(k+1)2^{-l}}\|_{[0,(k+1)2^{-l}]}\\
&\leq D(k) + KD(k)2^{-l} + \|I^m_{k2^{-l}} - I^m_{(k+1)2^{-l}}\|_{[0,(k+1)2^{-l}]}.
\end{split}
\end{align}
Let 
$$
\ds \beta_k = \frac{1}{n}\sum_{i=1}^n\sup_{0 \leq t \leq (k+1)2^{-l}}\big(f_i(t) + I^m(t)\big)^- - 
\frac{1}{n}\sum_{i=1}^n\sup_{0 \leq t \leq k2^{-l}}\big(f_i(t) + I^m(t)\big)^-,
$$
so that $K\beta_k$ is the amount the velocity $I^m$ increases in the interval 
$[k2^{-l}, (k+1)2^{-l}].$
From this telescoping definition of $\beta_k$ we see that 
\begin{align}
\sum_{k = 0}^{\lceil T2^l\rceil}\beta_k = \frac{1}{n}\sum_{i=1}^n\sup_{0 \leq s \leq T}\big(f(s) + I^m(s)\big)^- \leq \frac{1}{n}\|f\|_{[0, T]}.
\label{eq:betaBound}
\end{align}
Combine this with \eqref{eq:cauchyBound} above, we have
\begin{align}
\begin{split}
D(k+1) &\leq D(k) + KD(k)2^{-l} + K\beta_k2^{-l}\\
&= D(k) + K(D(k) + \beta_k)2^{-l}.
\end{split}
\label{eq:cauchyBound2}
\end{align}

\noindent
Set $A(k)$ to be recursively defined with the above inequality taken as 
equality. That is,
 \[
 A(k+1) = A(k) + KA(k)2^{-l} + K\beta_k2^{-l}.
 \]
We have $D(k) \leq A(k)$ for all $k.$ Note that $A(k)$ is maximized when
all the mass of $\sum_1^{\lceil T2^l\rceil}\beta_k$ is concentrated at $\beta_1$ because this 
allows the entire mass to be compounded from the beginning. Since the total sum of the $\beta_k$
does not exceed $\|f\|_{[0, T]}/n$, 
\begin{align*}
A(1) &= (K\|f\|_{[0, T]}/n)2^{-l},\\
A(2) &= A(1)(1 + K2^{-l}) + K\sum_{k = 1}^{2^l}\beta_k2^{-l}\\
&\leq A(1)(1 + K2^{-l}) + (K\|f\|_{[0, T]}/n)2^{-l},\\
A(k +& 1) = A(k)(1 + K2^{-l}).
\intertext{Such a recursion has an exponential bound:}
A(\lceil T2^l \rceil) &\leq (A(1) + A(2))(1 + K2^{-l})^{\lceil T2^l \rceil} \leq ((2+K)\|f\|_{[0, T]}/n)2^{-l}\exp{KT}.
\intertext{Hence,}
D(\lceil T2^{-l} \rceil) &\leq A(\lceil T2^l \rceil) \leq ((2+K)\|f\|_{[0, T]}/n)2^{-l}\exp{KT},
\end{align*}
which concludes the result.

\end{proof}

\noindent
To construct our system \eqref{eq:sysLaw} in the introduction from 
Proposition~\ref{prop:existence}, we apply the map $\Gamma_n$ pathwise
with 
\[
(f_1, \dots, f_n) = (B^{(1)} + X^{(n)}_1(0), \dots, B^{(n)} + X^{(n)}_n(0)) =: \vec{B} + \vec{X}(0),
\]
producing the pair of processes 
\[
\Gamma_n\,(\vec{B} + \vec{X}(0)) = \left(I^{(n)}_{\vec{B} + \vec{X}(0)}, \wt{V}^{(n)}_{\vec{B}+\vec{X}(0)}\,\right).
\]
Set 
\begin{align}
X^{(n)}_i = X^{(n)}_i(0) + B^{(i)} + m_i^{(n)}, \ V^{(n)} = -\wt{V}^{(n)}_{\vec{B} + \vec{X}(0)}
, \ Y^{(n)} = -I^{(n)}_{\vec{B} + \vec{X}(0)}.
\label{eq:YandVMap}
\end{align}
Then
\begin{prop}\label{prop:ex}
$(X_1^{(n)},\dots, X_n^{(n)}, Y^{(n)}, V^{(n)})$ satisfies \eqref{eq:sysLaw},
therefore giving a strong solution to that system of SDE's.
\label{prop:eqInLaw}
\end{prop}
\begin{proof}
We begin from \eqref{eq:sMap1} - \eqref{eq:sMap4} with the $f_i(t)$ replaced with
$B^{(i)}(t) + X^{(n)}_i(0).$ The following holds for all $t \in [0, T]$, almost surely:
\begin{align}
&X^{(n)}_i(t) = X^{(n)}_i(0) + B^{(i)}(t) + m_i^{(n)}(t) \geq Y^{(n)}(t), \label{eq:equivInLaw1}\\
&V^{(n)}(t) = v - \frac{K}{n}\sum_{i=1}^n m_i^{(n)}(t),\label{eq:equivInLaw2}\\
&Y^{(n)}(t) = \int\limits_0^t V^{(n)}(s)\, \mathrm{d}s,\label{eq:equivInLaw3}\\
&m_i^{(n)} \emph{ is flat off of } \{t : X^{(n)}_i(t) = Y^{(n)}(t)\}.
\label{eq:equivInLaw4}
\end{align}
We take $v = 0$ for convenience.
The fact that we have a strong solution of the system follows from the 
path-by-path construction.
We apply a transformation of measure argument.
As mentioned in Remark \ref{remark: sMap}, for a fixed time $t \in [0, T],$ 
\[
\ds V^{(n)}(t) = -\frac{K}{n}\sum_{i=1}^n\sup_{0 \leq u \leq t}\big(B^{(i)}(u) + X^{(n)}_i(0) - Y^{(n)}(u)\big)^-,
\]
which, due to nonegativity of $X^{(n)}_i(0)$ and the fact that $Y^{(n)} \leq 0,$  
$$\ds \sup_{u \in [0, T]}|V^{(n)}(u)| \leq \frac{K}{n}\sum_{i=1}^n\sup_{0 \leq u \leq T}\big(B^{(i)}(u)\big)^-.$$
This is equivalent to saying a continuous function plus a nonnegative drift has a running minimum below
zero less than that of the continuous function. 
It follows from continuity of the processes on $[0, T]$ that 
$\sup_{0 \leq u \leq t} |V^{(n)}(u)| \, \leq |V^{(n)}(T)| < \infty$ almost surely. Therefore  
$$\ds Z(t) = \exp\left(\frac{K}{n}\sum_{i=1}^n\int\limits_0^tV^{(n)}(s)\, \mathrm{d}B^{(i)}(s) - nY^{(n)}(t)\right)
$$
is a local martingale, and therefore there exists a collection
of exhaustive stopping times $\ds \tau_k \overset{a.s.}{\to} \infty$ such that 
$Z(t \land \tau_k)$ is a true martingale for each $k$. We will apply a Girsanov
transformation of measure, see \cite[Ch. 3.5]{KaratzasShreve}. Let $\mb{Q}$ be defined by $d\mb{Q}/d\mb{P} = Z(t \land \tau_k).$ Under $\mb{Q}$ each 
$\wt{B}^{(i)}(t \land \tau_k) := B^{(i)}(t \land \tau_k) - Y^{(n)}(t \land \tau_k)$ has 
the law of a Brownian motion, and 
the joint law of $(X^{(n)}_1, \dots, X^{(n)}_n)$, when stopped at $\tau_k$, has the same law as $\wt{X
}_i(t \land \tau_k) := X^{(n)}_i(0) + \wt{B}^{(i)}(t \land \tau_k) + \tilde{m}_i(t \land \tau_k) \geq 0,$ where 
$\tilde{m}_i(t) = \sup_{0 \leq u \leq t}\big(\wt{B}^{(i)}(u) + X^{(n)}_i(0)\big)^-.$
The $\tilde{m}_i$ are then equal to $m_i^{(n)} := \sup_{0 \leq u \leq t}\big(B^{(i)}(u) + X_i^{(n)}(0) - Y^{(n)}(u)\big)^-.$
Because $\tilde{m}_i$ is flat off $\{t : \wt{X}_i(t) = 0 \},$ the 
classical L\'{e}vy's theorem \cite[Chapter 3]{KaratzasShreve} shows that 
this system $(\wt{X}_1, \dots, \wt{X}_n)$ is equivalent in law to processes solving 
\[
d\wt{X}_i = d\wt{B}^{(i)} + d\wt{L}_i, \, \wt{X}_i(0) = X^{(n)}_i(0),
\] when stopped at $\tau_k$,
and where $\wt{L}_i$ is the local time at zero of $\wt{X}_i$. That is,
\begin{align*}
\wt{L}_i(t) &= \lim_{\e \to 0}\frac{1}{2\e}\int_0^t1_{[0, \e]}(\wt{X}_i(s))\, \mathrm{d}s \\
&= \lim_{\e \to 0}\frac{1}{2\e}\int_0^t1_{[0, \e]}(X_i^{(n)}(s) - Y^{(n)}(s))\, \mathrm{d}s =: L_i^{(n)}(t),
\end{align*}
for all $t$, almost surely. Additionally, 
$\tilde{m}_i$ is the local time of $\wt{X}_i^{(n)}$ at zero, which by definition
is the local time of contact between $X_i^{(n)}$ and $Y^{(n)}.$ Since $\tilde{m}_i = m_i^{(n)}$
this shows that $m_i^{(n)}(t) = L_i^{(n)}(t)$ for all $t$, almost surely.
This means that under $\prob,$ as processes stopped at $\tau_k$, solutions to
\eqref{eq:equivInLaw1}-\eqref{eq:equivInLaw3} are solutions of
\[
\ds dX^{(n)}_i = dB^{(i)} + dL^{(n)}_i, \ dY^{(n)} = -\frac{K}{n}\sum_{i = 1}^nL^{(n)}_i(t)dt,
\]
with the given initial conditions and
 where $L^{(n)}_i$ is the local time of $X^{(n)}_i - Y^{(n)}$ at zero. This latter
 process is the definition of \eqref{eq:sysLaw}. Since 
$\tau_k \to \infty$ almost surely, the equivalence in law holds as processes defined on $[0, T].$
\end{proof}

\begin{lemma}
Let $(Y^{(n)}, V^{(n)})$ be defined as in equation \eqref{eq:YandVMap}, then
\[
\|V^{(n)}\|_{[0, T]} \leq v + K\left(vT + \frac{1}{n}\sum_{i=1}^nm_i(T)\right),
\]
almost surely,
where $m_i(t) = \sup_{0 \leq s \leq t}\big(B^{(i)}(s)\big)^-.$
\label{lemma:boundOnV}
\end{lemma}
\begin{proof}
Clearly $\sup_{0 \leq s \leq T}V^{(n)}(s) \leq |v|,$ which implies that $Y^{(n)}(t) - Y^{(n)}(s) \leq |v|(t - s)$ for 
$(s, t) \subset [0, T].$ From Remark \ref{remark: sMap}, Lemma \ref{lemma:skorohodIneq}, and the fact that $X_i^{(n)}(0) \geq 0$, we have
\begin{align*}
\|V^{(n)}\|_{[0, T]} &= \sup_{0 \leq s \leq T}\left(|v| - \frac{K}{n}\sum_{i=1}^n\sup_{0 \leq u \leq s}\big(B^{(i)}(u) + X_i^{(n)}(0) - Y^{(n)}(u)\big)^-\right)\\
&\leq |v| + \frac{K}{n}\sum_{i=1}^n\sup_{0 \leq s \leq T}\big(B^{(i)}(s) - |v|s\big)^-\\
&\leq |v| + K\left(|v|T + \frac{1}{n}\sum_{i=1}^n\sup_{0 \leq s \leq T}\big(B^{(i)}(s)\big)^- \right).
\end{align*}
\end{proof}
\section{Hydrodynamic Limit and Propagation of Chaos}

\noindent
\\In subsection 3.1 we state the main lemmas, propositions, and corollaries used 
in the proofs of \ref{HL} and \ref{th: P-C}. Subsection 3.2 will contain the proofs of theorems \ref{HL} and \ref{th: P-C}. The proofs of the statements in subsection 3.1 will be given in subsection 3.3.
\subsection{Propositions and Lemmas}

\begin{lemma}
Let $(\Omega, \prob)$ be a probability space supporting a sequence of i.i.d.\ Brownian motions $B^{(i)}, i \in \N$ and initial conditions $\{X_i^{(n)}(0) : 1 \leq i \leq n, n \in \N\}$  independent from these Brownian motions and from one another. Define $V^{(n)}$ by the Skorohod map \eqref{eq:YandVMap}. For any $\e, \eta > 0$ there is a $\delta > 0$ such that
\[
\prob\big(\sup_n\omega(V^{(n)}, T, \delta) > \e \big) < \eta.
\]
\label{lemma:tightness}
\end{lemma}

\begin{cor}\label{cor:equi_V^n}
The collection $\{V^{(n)}(s) : s\in [0, T]\}$ is equicontinuous with probability 1. Furthermore,
for almost every $\omega,$ the sequence $\{V^{(n)}_\omega : n \in \N\}$ satisfies the Arzela-Ascoli criteria.
\end{cor}

\begin{prop}[Strong Convergence] \label{prop:strong_convergence}
In the same setting as Lemma \ref{lemma:tightness}, there is a continuous process $V$ such that
\[
V^{(n)} \lra V,
\] almost surely in the uniform norm on $C[0, T]$. Furthermore, 
\[
Y^{(n)} := \int_0^\cdot V^{(n)}(s) \, \md s \lra \int_0^\cdot V(s)\, \md s
\]
almost surely on $C[0, T]$.
\label{prop:uniquenessSubLimit}
\end{prop}

\begin{prop}
Let $(Y, V)$ be the almost sure limit of $(Y^{(n)}, V^{(n)})$, as guaranteed by Proposition
\ref{prop:strong_convergence}.
Then,
\begin{align*} 
&(Y, V) \text{ is deterministic,}\\ 
&V(t) = v - K\,\ex \, m(t), \text{ for all $t \in [0, T],$ where}\\
&m(t) = \sup_{0  \leq s \leq t}\big((B^{(1)}(s) + X_1^{(\infty)}(0)) - Y(s)\big)^-,
\end{align*}

 and $X_1^{(\infty)} \dist \pi_0(\mathrm{d}x).$
 \label{prop:subsequentialLimit}
\end{prop}
\begin{remark}
Since 
 $$\ds Y^{(n')}(t) = \int_0^tV^{(n')}(s)\, \mathrm{d}s$$
 for all $0 \leq t \leq T$, almost surely, we see that $Y^{(n')}$ converges uniformly to some deterministic 
 $Y$.
\end{remark}



The previous two propositions imply the following corollary. By an abuse of 
notation, we let $\mc{W}_r((Y^{(n)}, V^{(n)}), (Y, V))$ refers to the 
Wasserstein-$q$ distance between the measures induced by $(Y^{(n)}, V^{(n)})$
and $(Y, V)$ on $C([0, T], \R^2).$
\begin{cor}
There are deterministic functions $(Y(t), V(t))$ defined on $t \in [0, T],$ with $dY/dt = V,$
such that
\[
\mc{W}_r((Y^{(n)}, V^{(n)}), (Y, V)) \lra 0, \text{ for any $r \geq 1.$}
\]
\label{cor:convToY}
\end{cor} 
\noindent
In \cite{burdzy2004}, Burdzy, Chen and Sylvester study the density
of Brownian motion reflected inside a time dependent domain. They assume
the boundary is $C^3$ in both time and space, see 
\cite[Section 2]{burdzy2004}. In our case $n = 1,$ however,
their results still hold under the weaker assumption that the space-time boundary is $C^2.$ 
Let $g(t) \in C^2[0, T]$ be a twice differentiable function with $g(0) = 0$. 
Given a Brownian motion $B(t)$ and $x \geq 0$, let $p(t, y)$ be the 
transition density of the reflected Brownian motion solving 
$dX(t) = dB(t) + dL(t),$ and the
initial condition $X(0) = x,$ where $L$ is the local time of $X$ on $g.$ That is,
for a given Borel set $A \subset [g(t), \infty),$ 
$$\prob_x(X(t) \in A) = \int_Ap(t,y) \, \mathrm{d}y.$$

\begin{prop}[\cite{burdzy2004}, Theorem 2.4]\label{densityBounds}
For each fixed $l \in (0, T)$ there exists constants $K_l > 0$ and $C_l < \infty$
such that 
\[
p(t, y) \leq \frac{C_l}{\sqrt{t}}\exp\left( \frac{-K_l|x - y|^2}{t}\right)
\]
for all $0 < t < l$ and $(t, y) \in \{(s, z) : s \in [0, l], \ z \geq g(s)\}$ and
where $X(0) = x \geq 0 = g(0).$
\end{prop}

\begin{prop}[\cite{burdzy2004}, Theorem 2.9]\label{densitysolution}
The transition density $p(t, y)$ defined above solves the following 
heat equation in a time-dependent domain:
\begin{align*}
\begin{split}
&\frac{\partial p(t, y)}{\partial t} = \frac{1}{2}\Delta_yp(t, y), \ y > g(t),\\
&\frac{\partial^+ p(t, y)}{\partial^+ y} = -2g'(t)p(t, y), \ y = g(t),\\
&\lim_{t \downarrow 0}p(t, y)dy = \delta_{x}(\md y).
\label{eq:densityOfRBM}
\end{split}
\end{align*}
\label{prop:densityOfRBM}
\end{prop}
\begin{remark}
In particular, $p(t, y)$ has 
differentiability necessary for these statements to hold in the classical 
sense. See Remark \ref{remark:initialCondition} for an interpretation of the limiting statement.
\end{remark}
\begin{cor}
Let $\xi$ be a random variable with law $\pi_0(dx)$, independent from the Brownian motion $B
$, both supported on $(\Omega, \prob, \mc{F}_t).$ Let $g \in C^2([0, T], \R)$ and 
\[
X(t) = \xi + B(t) + m(t), \qquad m(t) = \sup_{0 < s < t}\big(\xi + B(s) - g(s)\big)^-.
\]
Then $p(t, x) := \prob(X(t) = dx)$ solves the PDE
\begin{flalign}
\begin{split}
&\frac{\partial p}{\partial t} = \frac{1}{2}\Delta_yp, \ y > g(t),\\
&\frac{\partial p}{\partial y} = -2g'(t)p, \ y = g(t),\\
&\lim_{t \downarrow 0}p(t, y)dy = \pi_0(dy).
\end{split}
\end{flalign}
\label{cor:densityOfRBM}
\end{cor}
\noindent
For a given time $0 < t < T$ and fixed value of $n$, the definition of our interacting diffusions 
gives us $n$ particles $X^{(n)}_1(t), \dots, X^{(n)}_n(t)$ which all lie in $[Y^{(n)}(t), \infty).$ 
Recall that
\begin{align}
\pi_t^{(n)} = \frac{1}{n}\sum_{i=1}^n\delta_{\{X^{(n)}_i(t)\}}
\label{eq:defOfPi}
\end{align}
denotes the empirical process of the arrangment of these particles. Similarly recall the definition
of $\mc{W}_p$ in \ref{def:Wp}
The main 
property of $\mc{W}_p$ we will need is
that $\mc{P}_p$ is separable and complete under $\mc{W}_p.$ Clearly $\pi^{(n)}_t$ is a random variable with 
state space $\mc{P}_p.$ In this way $\{\pi^{(n)}_t : t \in [0, T]\}$ is a $(\mc{P}_p, \mc{W}_p)-$valued stochastic 
process. By Lemma \ref{lemma:contPi} below $\pi^{(n)}_t$ is continuous,
and $(\pi^{(n)}_\cdot, Y^{(n)}(\cdot), V^{(n)}(\cdot))$ has the strong Markov property.
\begin{lemma}
For any collection $x_i, y_i \in \R, \ i = 1,\dots, n$ we have
$$
\mc{W}_p\Big(\frac{1}{n}\sum_{i=1}^n\delta_{\{x_i\}}, \frac{1}{n}\sum_{i=1}^n\delta_{\{y_i\}}\Big)
\leq \Big(\frac{1}{n}\sum_{i=1}^n|x_i - y_i|^p\Big)^{1/p}$$
\label{lemma:wpDiracBound}
\end{lemma}

\begin{lemma}
The pair $\{(\pi^{(n)}_t, Y^{(n)}(t), V^{(n)}(t)) : 0 \leq t \leq T\}$ is a continuous strong 
Markov process on $\mc{P}_p\times \R^2$ under the product metric $\mc{W}_p \times \|  \cdot  \|.$
\label{lemma:contPi}
\end{lemma}
\noindent
As $\pi^{(n)}$ is a continuous $\mc{P}_p-$valued process, it induces a 
probability measure on $C([0 ,T], (\mc{P}_p, \mc{W}_p)).$ 
We will abuse notation, which 
should be clear from context, by letting $\pi^{(n)}$ denote the measure on 
$C([0, T], \mc{P}_p)$, and $\pi^{(n)}_t$ to denote either the stochastic process or the 
element in $\mc{P}_p$ when $t$ is fixed. Let

\[
\tilde{\pi}^{(n)}_t := \frac{1}{n}\sum_{i=1}^n\delta_{\{\wt{X}^{(i)}(t)\}},
\]
where
\begin{align}
\label{eq:systemTilde}
d\wt{X}^{(i)} = dB^{(i)} + d\wt{L}^{(i)}, \qquad X^{(\infty)}_i(0) \overset{d}{=} \pi_0 \mbox{ for } i = 1, \dots, n,
\end{align}
the $X^{(\infty)}_i(0)$ are i.i.d.\ and $\wt{L}^{(i)}$ is the local
time of $\wt{X}^{(n)}_i$ on the function $Y$ given in Corollary
\ref{cor:convToY}.

\begin{prop}
There is a probability space supporting $\pi^{(n)}, \ \tilde{\pi}^{(n)}$ for all
$n$ such that
$$\sup_{0 \leq t \leq T}\mc{W}_p(\pi^{(n)}_t, \tilde{\pi}^{(n)}_t) \lra 0$$
almost surely.
\label{prop:equivalentLimitsPI}
\end{prop}
\begin{remark}
This implies distributional convergence
of $\pi^{(n)}$ and convergence of $\tilde{\pi}^{(n)}$ are equivalent. They
will approach the same limiting measure should one (hence both) of them converge.
\end{remark}

We use the following notions of modulus of continuity. For $\gamma \in C([0, T], (\mc{P}_p,$ $\mc{W}_p))$,
\[
\omega'(\gamma, T, \delta) = \sup_{\substack{0 \,\leq \,t \,\leq \,T \\ |t - s| \,< \,\delta}}\mc{W}_p(\gamma_t, \gamma_s),
\]
and similarly for $f \in C([0, T], \R),$
\[
\omega(f, T, \delta) = \sup_{\substack{0 \,\leq \,t \,\leq \,T \\ |t - s| \,< \,\delta}}|f(t) - f(s)|.
\]
We use
$p$th moment bounds of $\omega(B, T, \delta)$, for a Brownian motion $B$, in our proof of tightness for the collection $\{\pi^{(n)} : n \in \N\} \subset C([0, T], (\mc{P}_p(\R), \mc{W}_p))$. See \cite{FischerNappo}, where Fischer and Nappo provide such bounds in a more general setting.
These moment bounds can be contrasted with L\'{e}vy's theorem on the modulus of continuity for Brownian motion that
concerns the almost sure behavior of the modulus of continuity for small $\delta$.
\begin{theorem}[\cite{FischerNappo}]\label{th:moment_MOC}
Let $B(t)$ be a one dimensional Brownian motion and $T > \delta > 0.$ For any $q > 0$ there 
exists a positive constant $C_q$ independent of $T$ and $\delta$ such that
\[
\ex \, \omega(B, T, \delta)^q < C_q\left(\delta\log\frac{T}{\delta}\right)^{q/2}.
\]
\end{theorem}

\noindent
Theorem \ref{th:moment_MOC} directly implies a strong law of large numbers for
the modulus of continuity $\omega(B^{(i)}, T, \delta).$
\begin{cor}
Consider a sequence of independent Brownian motions $\{B^{(i)} : i \in \N\}$ all
defined on the same probability space. We have
$$\frac{1}{n}\sum_{i=1}^n\omega(B^{(i)}, T, \delta)^q \overset{a.s.}{\lra} 
\ex\, \omega(B^{(i)}, T, \delta)^q < C_q\left(\delta\log\frac{T}{\delta}\right)^{q/2}$$
for every $q > 0,$ every $\delta > 0,$ and some positive constant $C_q$ depending on $q$ only.
\label{cor:modulusSLLN}
\end{cor}

\begin{remark}\label{remark:equicont}
Typically when $X_n$ are continuous stochastic process on a complete and separable metric space
$(E, d)$, one demonstrates tightness of the measures induced on $C([0, T], E)$ by
showing ``stochastic equicontinuity''
\begin{align}
\lim_{\delta \to 0}\sup_{n}\prob(\omega(X_n, T, \delta) > \e) = 0
\label{eq:stochEquic}
\end{align}
together with a compact containment condition for a countable dense
set of times $[0, T]$: given any $\eta > 0$ one can find a relatively 
compact set $\Gamma_{t, \eta} \subset E$ such that
\begin{align}
\inf_n\prob(X_n(t) \in \Gamma_{t, \eta}) > 1- \eta.
\label{eq:CCC}
\end{align}
Consider \eqref{eq:stochEquic} and the corresponding $\delta$ for $\e = 1$. Repeated use of 
the triangle inequality between time increments of size $\delta$ can be used to bound $X_n(t
)$ with high probability uniformly in $n$ at each time $t$ should $X_n$ be bounded w.h.p.\ 
uniformly in $n$ at a fixed time $t_0$. Since boundedness in $\R^d$ is equivalent to relative compactness, if $E$ is Euclidean, \eqref{eq:CCC} can be concluded from \eqref{eq:stochEquic} provided there is some
time $t_0$ such that $X_n(t_0)$ is bounded w.h.p.\ uniformly in $n.$ If $E$ is not 
Euclidean, finding compact sets may not be
particularly easy, especially if $E$ is not locally compact.
Since our processes are $(\mc{P}_p, \mc{W}_p)-$valued continuous 
processes, as shown in Lemma \ref{lemma:contPi}, and since $(\mc{P}_p, \mc{W}_p)$ is not locally compact, 
we face similar issues. One can use the $p$th moment bounds on $\omega(B^{(i)}, T, \delta)$
 with similar arguments as in the proof of Proposition \ref{prop:unifEquicPi}
 to demonstrate \eqref{eq:stochEquic}. This would need to be paired with a compact containment condition as mentioned. We sidestep dealings with compact sets in $(\mc{P}_p, \mc{W}_p)$ by establishing almost sure 
 pointwise convergence of subsequential limits of $\pi^{(n)}_t$ together with a 
 uniform stochastic equicontinuity result Proposition \ref{prop:unifEquicPi} below.
\end{remark}
\begin{prop}
For every $\e, \eta > 0$ there exists a $\delta > 0$ such that
\[
\prob\,(\,\sup_n \omega'(\tilde{\pi}^{(n)}, T, \delta) \leq \e\,) > 1 - \eta.
\]
\label{prop:unifEquicPi}
\end{prop}

\begin{cor}
The collection $\{\tilde{\pi}^{(n)}, n \geq 1\}$ is equicontinuous on $C([0, T], (\mc{P}_p, \mc{W}_p))$ with probability 1.
\label{cor:almostSureEquicont}
\end{cor}
We quote one more theorem and present a lemma.
\begin{theorem}[\cite{fournierArnaud}]\label{W_p_Convergence}
For $q > p \geq 1, q \neq 2p,$ let $\{\xi_i : i \in \N\}$ be i.i.d.\ samples of an $L^q$ bounded random variable $\xi$
with density $f$,
all supported on the same probability space. Then
\[
\ex \, \mc{W}_p\Big(\frac{1}{n}\sum_{i=1}^n\delta_{\{\xi_i\}},  f\Big) \to 0, \text{ as } n \to \infty.
\]
\label{thm:WpConvofSampleMean}
\end{theorem}
\noindent

\begin{lemma}
Let $V$ be a continuous function, and $X$ a solution to $dX = dB + Vdt + dL$ where $L$
is the local time of $X$ at zero. Then 
$$
Z(t) = \exp\Big( -\int_0^tV_s\, \mathrm{d}B_s - \frac{1}{2}\int_0^tV_s^2\, \mathrm{d}s\Big)
$$
is a martingale with $Z(0) = 0$ and
$
\ex [Z(t)^p] < \infty
$
for any $p > 0.$
\label{lemma:expMartMoments}
\end{lemma}

\subsection{Proofs of Theorems \ref{HL} and \ref{th: P-C}}

We give the proof of Theorem \ref{th: P-C} first.
\begin{proof}[Proof of Theorem \ref{th: P-C}] We prove this assuming the results in subsection 3.1.
We give the details for two particles $X^{(n)}_1, X^{(n)}_2$.
The initial conditions $\xi_1$ and $\xi_2$ are independent by assumption.
Recall that $Y^{(n)}$ converges almost surely to $Y$ in $C[0, T].$ For $l = 1, 2$ set
\begin{align*}
m_l(t) &= \sup_{0 \leq u \leq t}\big(B^{(l)}(u)+ \xi_l - Y(u)\big)^-,\\
m_l^{(n)}(t) &= \sup_{0 \leq u \leq t}\big(B^{(l)}(u) + \xi_l - Y^{(n)}(u)\big)^-.
\end{align*}
By an argument similar to the one in Proposition \ref{prop:eqInLaw}, 
$\xi_l + B^{(l)} + m^{(n)}_l$ has the same distribution as $X^{(n)}_l, l = 1, 2.$ Since 
$Y^{(n)} \to Y$ almost surely, $m^{(n)}_l \to m_l$ almost surely as well.
Hence $\xi_l + B^{(l)} +  m^{(n)}_l \to \xi_l + B^{(l)} +  m_l$, almost surely. Clearly $ \xi_1 + B^{(
1)} + m_1$ and $\xi_2 + B^{(2)} + m_2$ are independent as each is a Brownian motion reflected
from $Y$, driven by different independent Brownian motions with independent initial positions.
\end{proof}
\noindent

\begin{proof}[Proof of Theorem \ref{HL}]\label{proofHL} We prove this assuming the results in subsection 3.1.
We first show that $\pi^{(n)}$ converges in distribution to the measure 
induced by $p(t, \cdot).$ By Proposition \ref{prop:equivalentLimitsPI} it
suffices to show this for $\tilde{\pi}^{(n)}.$ Take any
subsequence $n_k.$ 
For each rational $0 < t < T$ we have defined $\tilde{\pi}^{(n)}_t$ as
an empirical measure of i.i.d.\ random variables with density $p(t, \cdot)$ taken 
from Corollary \ref{cor:densityOfRBM} by replacing $g$ in the Corollary statement 
with $Y$. We will show later that $Y \in C^2[0, T]$ so this Corollary can be applied. By Theorem
\ref{thm:WpConvofSampleMean},
\[
\ex \, \mc{W}_p(\tilde{\pi}^{(n_k)}_t, p(t, \cdot)) \to 0, \text{ for each } t \in [0, T].
\]
For each rational $t \in [0, T]$
there is a subsequence $n_k'$ such that 
\[
\mc{W}_p(\tilde{\pi}^{(n_k')}_t, p(t, \cdot)) \to 0,
\]
almost surely. By a Cantor diagonalization applied to each subsequence
for an enumeration of the rationals, there exists a single subsequence $n_k''$
such that $\mc{W}_p(\tilde{\pi}^{(n_k'')}_t, p(t, \cdot))\to 0$ for each rational $t \in [0, T]$, almost surely.
Apply the uniform equicontinuity given by Corollary \ref{cor:almostSureEquicont}, and follow
the proof of Arz\'{e}la-Ascoli verbatim to see that the subsequence $\tilde{\pi}^{(n_k''
)}$ is totally bounded in the space $C([0, T], (\mc{P}_p, \mc{W}_p))$, almost surely. See \cite[Chapter 4.6]{Folland}. Total boundedness
in a metric space is equivalent to every sequence having a Cauchy subsequence. Consequently,
for almost every $\omega$ in the probability space, every subsequence of $\pi^{(n_k'')}(\omega)$ 
has a Cauchy subsequence in $C([0, T], (\mc{P}_p, \mc{W}_p)).$ Because $C([0, T], (\mc{P}_p, \mc{W}_p))$ is complete,
every subsequence of $\pi^{(n_k'')}(\omega)$ has a convergent subsequence. Since $\pi^{(n_k'')}_t(\omega)$ already converges to the continuous $p(t, \cdot)$ along rationals, every subsequence of
$\pi^{(n_k'')}(\omega)$ has a further subsequence converging to $p(t, \cdot).$ Therefore $\pi^{(n_k'')}_t$
converges to $p(t, \cdot)$ in $C([0, T], (\mc{P}_p, \mc{W}_p))$ almost surely. 
This proves the claim that $\{\pi^{(n)}_t : t \in [0, T]\}$ 
converges in distribution to $p(t, \cdot)$. Next, we show 
\[
V(t) = -\frac{K}{2}\int_0^tp(s, Y(s))\, \mathrm{d}s
\]
and that $V \in C^1[0, T]$. This also demonstrates $Y \in C^2[0, T]$, 
which we took for granted above. We take $v = 0$ for simplicity.
Let
\[
m(t) = \sup_{0 \leq u \leq t}\big(B^{(1)}(u) + X^{(\infty)}_i(0) - Y(u)\big)^-.
\]
As in the proof of Proposition \ref{prop:eqInLaw} we know $m(t)$ is 
distributed as $\wt{L}^{(1)}(t)$, the local time of 
\[
\wt{X}(t) := B^{(1)}(t) + X^{(\infty)}_1(0) + m(t)
\]
on $Y.$
From Corollary \ref{cor:convToY} we have, almost surely, 
\begin{align}
\begin{split}
V(t) &= \ex \, m_1(t) = \ex \, \wt{L}^{(1)}(t)\\
&=\ex \, \lim_{\e \to 0} \frac{-K}{2\e}\int_0^t1_{[0, \e]}(\wt{X}(s) - Y(s))\, \mathrm{d}s\\
&= \lim_{\e \to 0}\frac{-K}{2\e}\ex\int_0^t1_{[0, \e]}(\wt{X}(s) - Y(s))\, \mathrm{d}s\\
&= \frac{-K}{2}\lim_{\e \to 0}\int_0^t\frac{\text{F}(s, \e)}{\e}\, \mathrm{d}s,
\label{HL:exchangeE}
\end{split}
\end{align}
where F$(s, \e) = \prob(0 \leq \wt{X}(s) - Y(
s) \leq \e),$ provided we justify the 
passing of the limit under the expectation. 
In the proof of Proposition \ref{prop:eqInLaw} we saw that $\wt{X} - Y$ solves an SDE of the form
$dW = dB + Vdt + dL$ for the continuous function $V,$ and $L$ is the local time of $W$ at zero. By Propositions \ref{densityBounds} and
\ref{densitysolution} such density $\phi(s, x) = p(s, Y(s) + x)$ exists and has an upper bound. This allows us to write
\begin{align*}
&\frac{1}{\e}\int_0^t\text{F}(s, \e)\, \mathrm{d}s = \int_0^t\frac{1}{\e}\int_0^\e\phi(s,x)\, \mathrm{d}x\,\mathrm{d}s
\end{align*} 
By Proposition \ref{densityBounds} we have 
\[
\frac{1}{\e}\int_0^\e\phi(s, x) \, \md x \leq \frac{C_l}{\sqrt{s}},
\] which is integrable for $s$ is a compact time set. By the fundamental theorem, taking limits as $\e \to 0$ yields
\begin{align*}
&\frac{-K}{2}\lim_{\e \to 0}\frac{1}{\e}\int_0^t\text{F}(s, \e)\, \mathrm{d}s = \int_0^t\phi(s, 0)\, \md s
\end{align*}
and the bounded convergence theorem justifies the passing of the limit inside the
time integral,
$$
\frac{-K}{2}\lim_{\e \to 0}\int_0^t\frac{\text{F}(s, \e)}{\e}\, \mathrm{d}s 
= \frac{-K}{2}\int_0^t\lim_{\e \to 0}\frac{\text{F}(s, \e)}{\e}\, \mathrm{d}s 
= \frac{-K}{2}\int_0^t\phi(s, 0)\, \mathrm{d}s.
$$
That is,
$$
V(t) = \frac{-K}{2}\int_0^t p(s, Y(s))\, \mathrm{d}s.
$$
We now justify the exchange of limit in \eqref{HL:exchangeE} using the
definition of local time to replace the time integral with a space
integral. Let $\wt{L}^{(1)}(s, a)$ denote the local time of $\wt{X} - Y$ at level $a$ and
time $s.$ We see
\begin{align*}
\frac{1}{\e}\int_0^t1_{[0, \e]}(X(s) - Y(s))\, \mathrm{d}s 
&= \int_0^t\frac{1}{\e}\int_0^\e \wt{L}^{(1)}(s, z)\, \mathrm{d}z\,\mathrm{d}s
\leq \int_0^t \sup_{z} \, [\wt{L}^{(1)}(s, z)]\, \mathrm{d}s.
\end{align*}
The Lebesgue dominated convergence theorem will justify
\eqref{HL:exchangeE} provided we show
$$\ds \ex \int_0^t \sup_{z}[\wt{L}^{(1)}(s, z)]\, \mathrm{d}s \leq t \,\ex \sup_z \, [\wt{L
}^{(1)}(t, z)] <\infty.$$ We apply a Girsanov change of measure as in Lemma 
\ref{lemma:expMartMoments} which is justified because $Y$ satisfies the Novikov condition. So
$$
Z(t) = \exp\Big( -\int_0^tV_s\, \mathrm{d}B_s - \frac{1}{2}\int_0^tV_s^2\, \mathrm{d}s\Big)
$$
is an exponential martingale with $|B|$ 
having the same distribution as $\wt{X} - Y$ under the measure defined by the Girsanov transformation
$\ds d\mb{Q}/d\prob = Z(t)$.
Lemma \ref{lemma:expMartMoments} states $\ex [Z(T)^2] = C < \infty.$ From this, the change of measure formula, and Cauchy-Schwarz,


\begin{align*}
\ex \, \sup_{z}[\wt{L}^{(1)}(t, z)] &= \ex( Z(t) \sup_{z}L(t, z)) \\
&\leq \ex(Z^2(t))^{1/2} \, \ex \, [(\sup_{z}L(t, z))^2]^{1/2} \\
&\leq C^{1/2} \, \ex \, [(\sup_{z}L(t, z))^2]^
{1/2}
\end{align*}
where $L(t, z)$ is the local time at level $z$ of Brownian motion reflected from the origin.
The main results in \cite[Theorem 3.1]{barlow1981semi} demonstrate bounds on the last term, 
where Barlow and Yor show the existence of a constant $C_p$ such that
\[
\ex \,[(\sup_{z}L_t(z))^p] \leq C_p \, t^{p/2}.
\]
It follows that
$$
\ex \, \sup_{z}[\wt{L}^{(1)}(t, z)] < \infty,
$$
completing the proof.
\end{proof}
\subsection{Proofs of Lemmas and Propositions}
\begin{proof}[Proof of Lemma \ref{lemma:tightness}]
We take $v = 0$ for convenience. By our representation of
$V^{(n)}$ as \eqref{eq:sMap3}, together with Remark \ref{remark: sMap},
we apply Lemma \ref{lemma:skorohodIneq} with $y_1 = -Y^{(n)}$ and $y_2 = 0$ to show
the increment $V^{(n)}(t + \delta) - V^{(n)}(t)$ is not more than the change
of the sample average of the signed running minimums of the $n$ Brownian paths. 

That is, letting $m_1(t), \dots, m_n(t)$ denote the respective signed running minimum
below zero of $B^{(1)}, \dots, B^{(n)},$
\[
\ds 0 \leq |V^{(n)}(t + \delta) - V^{(n)}(t)
| \leq \frac{K}{n}\sum_{i=1}^n(m_i(t + \delta) - m_i(t)) \text{ for all } t \in [0, T - \delta],
\]
for any positive $\delta$, almost surely. Since $V^{(n)}$ and $m_i$ are a.s. nondecreasing, we have the
same inequality but for the modulus of continuity:
\begin{align}
\omega(V^{(n)}, T, \delta) \leq \frac{K}{n}\sum_{i=1}^n\omega(m_i, T, \delta) =: S_n(\delta),
\label{abound}
\end{align}
almost surely. For every $p > 0$, there is a $C(p) > 0$ such that 
\[
\ex\,  \omega(m_i, T, \delta) \leq C \delta\log\left( \frac{T}{\delta}\right)^{1/2}.
\] (See Theorem \ref{th:moment_MOC}). Because $\omega(m_i, T, \delta)$ are i.i.d., the strong law of large numbers implies
\[
\limsup_{n \to \infty}\omega(V^{(n)}, T, \delta) \leq C\delta\log\left( \frac{T}{\delta}\right)^{1/2}.
\]
Consequently,
\[
\lim_{\delta \to 0}\limsup_{n \to \infty}\omega(V^{(n)}, T, \delta) = 0,
\]
almost surely. This is sufficient for the statement of the Lemma.
\end{proof}

\begin{proof}[Proof of Corollary \ref{cor:equi_V^n}]
Let $\e_k = 1/k, \eta_k = 2^{-k}$. By Lemma \ref{lemma:tightness} one has a corresponding sequence $\delta_k > 0$ such that
\[
\sum_{k = 1}^\infty\prob\big(\sup_n\omega(V^{(n)}, T, \delta_k) > 1/k \big) < \infty.
\]
So by Borel-Cantelli, $\big(\sup_n\omega(V^{(n)}, T, \delta_k) \leq 1/k \big)$ occurs for all but finitely many $k$, almost surely. Since $V^{(n)}(0) = v$ almost surely, we conclude the sequence
$V^{(n)}$ satisfies Arzela-Ascoli's criteria almost surely.
\end{proof}

\begin{proof}[Proof of Proposition \ref{prop:uniquenessSubLimit}]
By Corollary \ref{cor:equi_V^n} for almost every $\omega$ in our probability space, the 
sequence $V^{(1)}_\omega, V^{(2)}_\omega, \dots,$ satisfies the Arzela-Ascoli criteria. Let
$n_k^1(\omega)$ and $n_k^2(\omega)$ be two (random) subsequences such that $V^{(n_k^1)}_\omega, V^{(n_k^1)}_\omega$ converge to $V^1_\omega$ and $V^2_\omega$ respectively. We drop the $\omega$ for convenience. In the same way let 
\[
Y^1(t) = \int_0^t V^1(s)\, \md s, Y^2(t) = \int_0^t V^1(s)\, \md s
\] be the two limits associated with $n^1_k, n^2_k.$ That is, 
\[
\ds \lim_{n_k^i \to \infty} Y^{(n_k^i)} = Y^i
\] for $i=1, 2.$ By our pathwise construction given in Theorem \ref{prop:existence},
\begin{align}
Y^i = -\lim_{n_k^i \to \infty}\lim_{2^{-l} \to 0}I^{(n_k^i,2^{-l})}_{\big(B^{(1)} + X^{(n_k^i)}_1(0), \dots, B^{(n_k^i)} + X^{(n_k^i)}_{n_k^i}(0)\big)}.
\label{eq:doubleLimit}
\end{align}
It follows from the strong law of large numbers and the $\mc{W}_q$-assumption on $X_i^{(n)}(0)$ 
that for almost every $\omega$ there is a $C(\omega) < \infty$ such that
\[
\frac{1}{n}\|(B^{(1)} + X^{(n)}_1(0), \dots , B^{(n)}+ X^{(n)}_n(0))\|_{[0
, T]} < C(\omega) < \infty.
\]
Applying Proposition \ref{prop:cauchy} we see that 
\[
\|I^{(n_k^i, 2^{-l})} - I^{(n_k^i, 2^{-m})}\|_{[0, T]} \leq (2+K)C(\omega)2^{-l}\exp(KT).
\]
Let $m \to \infty$ and we have
\[
\|I^{(n_k^i, 2^{-l})} - I^{(n_k^i)}\|_{[0, T]} \leq (2+K)C(\omega)2^{-l}\exp(KT).
\]
In other words,
\[
\sup_{n_k^i \geq 1, i = 1, 2}\|I^{(n_k^i, 2^{-l})} - I^{(n_k^i)}\|_{[0, T]} \leq (2+K)C(\omega)2^{-l}\exp(KT),
\] and as $2^{-l} \to 0$ the convergence of $I^{(n_k^i, 2^{-l})}$ to $I^{(n_k^i)}$
is uniform over $(n_k^i)_{k \geq 1},$ almost surely. This is sufficient  to guarantee an interchange of
the limiting operations in \eqref{eq:doubleLimit}. See \cite[Th. 6 VII.3]{lang}. Hence,
\begin{align*}
&Y^i = -\lim_{2^{-l} \to 0}\lim_{n_k^i \to \infty}I^{(n_k^i, 2^{-l})}.
\end{align*}
We will use the strong law of large numbers to show $\lim_{n_k^1 \to \infty}I^{(n_k^1, 2^{-l})} =
\lim_{n_k^2 \to \infty}I^{(n_k^2, 2^{-l})}.$ This can be seen by induction
on $[0, N2^{-l}]:$
By construction of the $I^{(n_k^i, 2^{-l})}$ the two limits are identically zero on $[0, 2^{-l}]$. 
Assume the two limits agree on $[0, N2^{-l}].$
This induction hypothesis implies the slope of $I^{(n_k^1, 2^{-l})}$ and the slope of
$I^{(n_k^2, 2^{-l})}$ become arbitrarily close as $k \to \infty.$ Since the slope of $I^{(n_k^i, 2^{-l})}$ 
on $[N2^{-l}, (N+1) 2^{-l}]$ is the average of the positive part of the
running minimums of $B^{(1)} + I^{(n_k^i, 2^{-l})}, \dots , B^{(n_k^i)} + I^{(n_k^i, 2^{-l})},$
and because the limit in the strong law of large numbers is independent on the subsequence chosen, the
slopes of $I^{(n_k^1, 2^{-l})}, I^{(n_k^2, 2^{-l})}$ become arbitrarily close on $[0, (N+1)2^{-l}]$ as
$k \to \infty$. This completes the induction step. 
\end{proof}
\begin{proof}[Proof of Proposition \ref{prop:subsequentialLimit}]
Let $(\Omega, \prob, (\mc{F}_t)_{t \geq 0})$ be a probability space supporting a sequence of independent $\mc{F}_t$-adapted 
Brownian motions $\{B^{(i)} : i \in \N\}.$ 
Define $X_i^{(n)}, Y^{(n)}, V^{(n)}$ by \eqref{eq:YandVMap}. Recall that the initial 
conditions $X^{(n)}_i(0)$, $1 \leq i \leq n,$ are i.i.d.\ samples with distribution
$\pi^{(n)}_0$ and that $\mc{W}_q(\pi^{(n)}_0, \pi_0) \to 0$ by assumption.
Let $X^{(\infty)}_i(0)$, $i \in \N,$ be independent samples with distribution $\pi_0.$
By definition of the $\mc{W}_p$ metric we enlarge this probability space to support all the processes 
$\{X^{(n)}_i(t) : t \in [0, T], 1 \leq n, n \in \N\}$ such that
\[
\sup_{1 \leq i \leq n}\ex|X_i^{(n)}(0) - X^{(\infty)}_i(0)|^q \to 0, \text{ as } n\to \infty.
\]
Hence, 
\[
\sup_{1 \leq i \leq n}\ex |X_i^{(n)}(0) - X^{(\infty)}_i(0)| \leq \sup_{1 \leq i \leq n}\left(\ex |X_i^{(n)}(0) - X_i^{(\infty)}(0)|^q\right)^{1/q} \to 0,
\]
as $n\to \infty.$ This enlarged 
filtration can be constructed by taking $\mc{F}_t \times \sigma \{X^{(n)}_i(0) : 1 \leq i \leq n, \, n \in \N\}$ as our new filtration. Then,
\[
\frac{1}{n}\ex\sum_{i=1}^n|X^{(n)}_i(0) - X^{(\infty)}_i(0)| \leq \sup_{1 \leq i \leq n} \ex|X^{(n)}_i(0) - X^{(\infty)}_i(0)| \lra 0,
\]
and so
\[
\frac{1}{n}\sum_{i=1}^n|X^{(n)}_i(0) - X^{(\infty)}_i(0)| \overset{\mb{P}}{\lra} 0.
\] 
As a consequence of this convergence in probability to zero, every sequence $n'$ has a further subsequence $n_k'$ with 
\begin{align}
\frac{1}{n_k'}\sum_{i=1}^{n_k'}|X^{(n_k')}_i(0) - X^{(\infty)}_i(0)| \lra 0,
\label{eq:convOfInitialCond}
\end{align} almost surely. Without loss of generality we relabel such a sequence $n_k'$
as $n'.$
We apply Proposition \ref{prop:contV} with $f = (B^{(1)} + X^{(n')}_1(0), \dots, B^{(n')} + X^{(n')}_{n'}(0))$, $g = (B^{(1)} + X^{(\infty)}_1(0), \dots, B^{(n')} + X^{(\infty)}_{n'}(0))$:
\begin{align*}
&\big\|V^{(n')}_{(B^{(1)} + X^{(n')}_1(0), \dots, B^{(n')} + X^{(n')}_{n'}(0))}
- V^{(n')}_{(B^{(1)} + X^{(\infty)}_1(0), \dots, B^{(n')} + X^{(\infty)}_{n'}(0))}\big\|_{[0, T]}\\
&\leq \frac{K}{n'}\sum_{i=1}^{n'}|X_i^{(n')}(0) - X_i^{(\infty)}(0)|\exp (KT) \to 0
\end{align*}
almost surely. Therefore it suffices to show $V^{(n')}_{(B^{(1)} + X^{(\infty)}_1(0), \dots, B^
{(n')} + X^{(\infty)}_{n'}(0))}$ converges to a deterministic limit.

Fix $k \in \N.$ For almost all $\omega$ in our probability space there is a constant $C(\omega, k)$ such that 
$\|(B^{(1)} + X^{(\infty)}_1(0), \dots, B^{(k)} + X^{(\infty)}_k(0))\|_{[0, T]} < C(\omega, k) < \infty$. This follows from continuity of the $B^{(i)}$ and the assumption that the initial 
samples $X^{(\infty)}_i(0)$ come from an almost surely finite, in fact, an $L^p$ bounded,
random variable.
Apply Proposition \ref{prop:contV} with 
\begin{align*}
f &= (B^{(1)} + X^{(\infty)}_1(0), \dots, B^{(n')} + X^{(\infty)}_{n'}(0)),\\
g &= (0, \dots, 0, B^{(k + 1)} + X^{(\infty)}_{k + 1} B^{(n')} + X^{(\infty)}_{n'}(0)),
\end{align*} i.e.\ the first $k$ coordinates of $g$ are zero, and
$\eta = \|f - g\|_{[0, T]} < C(\omega , k)$ to give the almost sure bound
 
\begin{align*}
&\|V^{(n')}_{(B^{(1)} + X^{(\infty)}_1(0), \dots, B^{(n')} + X^{(\infty)}_{n'}(0))}- V^{(n')}_{(0, \dots 0, B^{(k+1)} + X^{(\infty)}_{k+1}(0), \dots, B^{(n')} + X^{(\infty)}_{n'}(0))}\|_{[0, T]}\\
&\leq (KC(\omega, k)/n')\exp(KT) \to 0 \text{ as } \ n' \to \infty.
\end{align*}
Therefore 
$$V = \lim_{n' \to \infty} V^{(n')}(0, \dots 0, B^{(k+1)} + X^{(\infty)}_{k+1}(0), \dots, B^{(n')} + X^{(\infty)}_{n'}(0)) \in \mathcal{F}_T^{k+1, \infty}$$
where $\mathcal{F}^{k, \infty}_T$ is the sigma-field generated by $\{B^{(i)}(t) + X^{(\infty)}_i(0): 0 \leq t \leq T, k \leq i \}$.
By definition this means the continuous function $V$ is adapted to the tail 
sigma-field of the infinite sequence of i.i.d.\ processes. Hence $\{V(t) : t \in [0, T] \}$ 
is adapted to a trivial sigma-field, so it is deterministic. Consequently $Y$ is also deterministic.
\\

To prove the equalities in the proposition statement, we take $v = 0$ for simplicity. For $i = 1, \dots, n',$ define
\begin{align*}
&m_i^{(n)}(t) = \sup_{0 \leq u \leq t}\big((B^{(i)}(u) + X^{(n)}_i(0)) - Y^{(n)}(u)\big)^-,\\
&\tilde{m}_i(t) = \sup_{0 \leq u \leq t}\big((B^{(i)}(u) + X^{(\infty)}_i(0)) - Y(u)\big)^-.
\end{align*}
Since
\[
V^{(n')}(t) = -\frac{K}{n}\sum_{i=1}^{n'}m_i^{(n')}(t),
\]
we compute
\begin{align}
\begin{split}
&\Big\|\frac{K}{n'}\sum_{i=1}^{n'}\tilde{m}_i - V^{(n')}\Big\|_{[0, T]} \leq 
\frac{K}{n'}\sum_{i=1}^{n'}\|\tilde{m}_i - m_i^{(n')}\| \\
&\leq \frac{K}{n'}\sum_{i=1}^{n'}\Big(|X^{(\infty)}_i(0) - X^{(n')}_i(0)| + \|Y^{(n')} - Y\|_{[0, T]}\Big)\\
& \lra 0
\label{eq:tightnessV}
\end{split}
\end{align}
almost surely. In words, $V^{(n')}$ and the average of the running minimum of the i.i.d.\
Brownian paths below the curve $Y$ become arbitrarily close in the uniform distance. By the strong law of large numbers $\frac{1}{n'}\sum_{i=1}^{n'}\tilde{m}_i(t) \to \ex \, \tilde{m}_i(t)$ almost surely for each $t$. That is,
\begin{align}
\lim_{n' \to \infty}V^{(n')}(t) = -\lim_{n' \to \infty}\frac{K}{n'}\sum_{i=1}^{n'}\tilde{m}_i(t) = -K\ex\,\tilde{m}_i(t), \text{ almost surely.}
\label{eq:VisMeanOfM}
\end{align}
By \eqref{eq:tightnessV}, $V^{(n')}(\cdot)$ converges in the uniform norm to $-K\ex \, \tilde{m}_i(\cdot)$, almost surely.
\end{proof}

\begin{proof}[Proof of Corollary \ref{cor:convToY}]
Convergence in $\mc{W}_r$ for any $r \geq 1$ is shown once we can establish that $Y^{(n)}$ and $V^{(n)}$ converge almost surely and in $L_r$ to $Y$ and $V$, respectively, in
some probability space supporting a sequence of i.i.d.\ Brownian motions and the
initial conditions $\{X_i^{(n)}(0) : 1 \leq i \leq n, m \in \N\}$.
We already know $Y^{(n)} \lra Y$ almost surely. The convergence in $\mc{W}_r$ comes from the bound indicated in
the proof of Proposition \ref{prop:eqInLaw}, that 
\[
\ds |V^{(n)}(t)| \leq \frac{K}{n}\sum_{i=1}^nL_i'(t)
\]
where the $L^i$ are i.i.d.\ local times at zero of Brownian motion. Now use 
the fact that 
$\frac{1}{n}\sum_{i=1}^nL_i'$ converges almost surely and
in $L_r$ to its mean function and apply the (generalized) dominated convergence
theorem \cite[Chapter 2.3]{Folland}. 
\end{proof}
\begin{proof}[Proof of Corollary \ref{cor:densityOfRBM}]
As in the proof of Proposition \ref{prop:eqInLaw}, it follows from L\'{e}vy's theorem applied after
a Girsanov change of measure that $X$ is distributed as a Brownian motion reflected from the curve $g.$
Now apply Proposition \ref{prop:densityOfRBM} after conditioning on $\xi.$ See Remark \ref{remark:initialCondition} for a brief discussion of the last condition.
\end{proof}

\begin{proof}[Proof of Lemma \ref{lemma:wpDiracBound}]
This follows from coupling $(X, Y)$ with 
\[
\ds X \overset{d}{=} \frac{1}{n}\sum_{i=1}^n\delta_{\{x_i\}}, \ Y \overset{d}{=} \frac{1}{n}\sum_{i=1}^n\delta_{\{y_i\}}
\]
so that
 $X$ has mass on $\{x_i\}$ exactly when $Y$ has mass on $\{y_i\}.$
\end{proof}

\begin{proof}[Proof of Lemma \ref{lemma:contPi}]
The strong Markov property follows from the strong Markov property of \((X^{(n)}_1, \dots , X^{(n)}_n, Y^{(n)}, V^{(n)})\).
We need only show continuity of $\pi^{(n)}$ since $(Y^{(n)}, V^{(n)})$ is continuous.
By Lemma \ref{lemma:wpDiracBound},
\begin{align}
\mc{W}_p(\pi^{(n)}_t, \pi^{(n)}_s) \leq 
\Big(\frac{1}{n}\sum_{i=1}^n|X^{(n)}_i(t) - X^{(n)}_i(s)|^p\Big)^{1/p},
\label{eq: wpDistance}
\end{align}
and continuity follows from the continuity of the $X_i^{(n)}.$
\end{proof}

\begin{proof}[Proof of Proposition \ref{prop:equivalentLimitsPI}]
Consider the probability space supporting all the $\{B^{(i)}(t) : 0 \leq t \leq T\}$ together with
the initial conditions $\{X^{(n)}_i(0) : 1 \leq i \leq n, n \in \N\}.$
This space will then support $Y^{(n)}, Y$ as well. By Corollary \ref{cor:convToY} we may also assume $Y^{(n)} \to Y$
almost surely. As in the proof of Proposition \ref{prop:subsequentialLimit}, $\{X^{(\infty)}_i(0) : \, i \in \N\}$ are i.i.d.\ samples with distribution $\pi_0.$ By our assumption
that $\pi^{(n)}_0 \to \pi_0$ in $(\mc{P}_p, \mc{W}_p)$, we may further choose our probability space so that
\begin{align}
\frac{1}{n}\sum_{i=1}^n|X^{(n)}_i(0) - X^{(\infty)}_i(0)|^p \to 0
\label{eq:initialSampleConv}
\end{align}
in probability. Using the Skorohod representation theorem we can find a supporting probability space where this holds almost surely. We use the same representation of our processes as in the proof of the
propagation of chaos. That is,
\begin{align}
X^{(n)}_i(t) &= X^{(n)}_i(0) + B^{(i)}(t) + m_i^{(n)}(t),\\
\wt{X}^{(i)}(t) &= X^{(\infty)}_i(0) + B^{(i)}(t) + \tilde{m}_i(t),\\
\intertext{ for $i = 1, \cdots, n,$ and $t \in [0, T]$, where}
&m_i^{(n)}(t) = \sup_{0 \leq u \leq t}\big((B^{(i)}(u) + X^{(n)}_i(0)) - Y^{(n)}(u)\big)^-,\\
&\tilde{m}_i(t) = \sup_{0 \leq u \leq t}\big((B^{(i)}(u) + X^{(\infty)}_i(0)) - Y(u)\big)^-.
\end{align}
By the triangle inequality
 \begin{align}
 \| m_i^{(n)} - \tilde{m}_i\|_{[0, t]} \leq |X_i^{(n)}(0) - \wt{X}_i^{(\infty)}(0)| +
 \|Y^{(n)} - Y\|_{[0, t]}
 \label{eq:eq1}
 \end{align} for any $t \in [0, T].$ For any nonnegative numbers $a$ and $b,$
 $(a + b)^p \leq (2\max\{a, b\})^p \leq 2^p(a^p + b^p).$ Using \eqref{eq:eq1}, Lemma \ref{lemma:wpDiracBound}, \eqref{eq:initialSampleConv} and the fact that $\|Y^{(n)} - Y\|_
{[0, T]} \to 0$ almost surely,
\begin{align*}
\sup_{0 \leq t \leq T}\mc{W}_p(\pi^{(n)}_t, \tilde{\pi}^{(n)}_t) 
&\leq \sup_{0 \leq t \leq T}\Big(\frac{1}{n}\sum_{i=1}^n|X^{(n)}_i(t) - \wt{X}^{(i)}(t
)|^p\Big)^{1/p}\\
&\leq \sup_{0 \leq t \leq T}\Big(\frac{1}{n}\sum_{i=1}^n\Big(|X_i^{(n)}(0) - X^{(\infty)}_i(0)| + 
\|m_i^{(n)} - \tilde{m}_i\|_{[0, t]}\Big)^p\Big)^{1/p}\\
&= \Big(\frac{1}{n}\sum_{i=1}^n2^{p+1}|X^{(n)}_i(0) - X^{(\infty)}_i(0)|^p + 2^p\|Y^{(n)} - Y\|_{[0, T]}^p\Big)^{1/p}\\
& \lra 0,
\end{align*}
almost surely. 
\end{proof}

\begin{proof}[Proof of Proposition \ref{prop:unifEquicPi}]
Recall the role of $v$ in \eqref{eq:sysLaw}. From Lemma \ref{lemma:wpDiracBound} and the definitions of $\omega, \omega'$ the following holds almost surely,
\begin{align*}
\omega'(\tilde{\pi}^{(n)}, T, \delta)
&:= \sup_{\substack{0 \, \leq \, t \, \leq \, T\\ |t -s| \, < \delta}}\, \mc{W}_p(\tilde{\pi}^{(n)}_s, \tilde{\pi}^{(n)}_t)\\
&\leq \sup_{\substack{0 \, \leq \, t \, \leq \, T\\ |t -s| \, < \delta}}\,\Big(\frac{1}{n}\sum_{i=1}^n\left[|B^{(i)}(s) - B^{(i)}(t)| + |\tilde{m}_i(s) - \tilde{m}_i(t)|\right]^p\Big)^{1/p}\\
 &\Big(\frac{1}{n}\sum_{i=1}^n \sup_{\substack{0 \, \leq \, t \, \leq \, T\\ |t -s| \, < \delta}}
\left[|B^{(i)}(s) - B^{(i)}(t)| + |\tilde{m}_i(s) - \tilde{m}_i(t)|\right]^p\Big)^{1/p}\\
&\leq \Big(\frac{2^p}{n}\sum_{i=1}^n \sup_{\substack{0 \, \leq \, t \, \leq \, T\\ |t -s| \, < \delta}}|B^{(i)}(t) - B^{(i)}(s)|^p \, + \sup_{\substack{0 \, \leq \, t \, \leq \, T\\ |t -s| \, < \delta
}} |\tilde{m}_i(t) - \tilde{m}_i(s)|^p\Big)^{1/p}
\end{align*}
Where we moved the supremum inside and used the bound $(a + b)^p \leq 2^p(a^p + b^p)$. The above is less than or equal to
\begin{align*}
&\Big(\frac{2^p}{n}\sum_{i=1}^n \sup_{\substack{0 \, \leq \, t \, \leq \, T\\ |t -s| \, < \delta}}|B^{(i)}(t) - B^{(i)}(s)|^p \, + \sup_{\substack{0 \, \leq \, t \, \leq \, T\\ |t -s| \, < \delta
}} |\tilde{m}_i(t) - \tilde{m}_i(s)|^p\Big)^{1/p}\\
&= \Big(\frac{2^p}{n}\sum_{i=1}^n \omega(B^{(i)}, T, \delta)^p + \omega(\tilde{m}_i, T, \delta)^p
\Big)^{1/p}.
\end{align*}
Because $dY/dt \leq |v|$, applying Lemma \ref{lemma:skorohodIneq} with $y_1 = -Y$ and $y_2(t) = |v|t$, we have 
\[
\omega(\tilde{m}_i, T, \delta) \leq \omega(B^{(i)} + y_2, T, \delta) \leq |v|\delta + \omega(B^{(i)},T, \delta).\] That is, the maximum change the Brownian path makes below $Y$ until time $T$, in the span of
 $\delta$ time, is bounded by the change made by the line $|v|t$ in addition to the change of the Brownian path. 
This gives
\begin{align*}
&\omega'(\tilde{\pi}^{(n)}, T, \delta) \leq \Big(2^pv\delta + \frac{2^{p+1}}{n}\sum_{i=1}^n \omega(B^{(i)}, T, \delta)^p\Big)^{1/p}
\end{align*}
almost surely. For simplicity we take $v = 0$ in the remaining argument. Setting $I_\e = (\e^p/2^{p+1}, \infty),$
\begin{align*}
&\prob\big(\sup_{n > N} \omega'(\tilde{\pi}^{(n)}, T, \delta) > \e\big)\\
&\leq \prob\Big(\sup_{n > N}\frac{1}{n}\sum_{i=1}^n\omega(B^{(i)}, T, \delta)^p > \frac{\e^p}{2^{p+1}}\Big)\\
&= \mathlarger{\ex} \, \mathlarger{1}_{I_{\e}}\Big\{{\sup_{n > N}\frac{1}{n}\sum_{i=1}^n\omega(B^{(i)}, T, \delta)^p}\Big\}.
\end{align*}
By Corollary \ref{cor:modulusSLLN} and the dominated convergence theorem,
\begin{align*}
&\lim_{N \to \infty}\mathlarger{\ex} \, \mathlarger{1}_{I_\e}\Big\{{\sup_{n > N}\frac{1}{n}\sum_{i=1}^n\omega(B^{(i)}, T, \delta)^p}\Big\}\\
&= \ex 1_{I_\e}\Big\{\ex \, \omega(B^{(i)}, T, \delta)^p \Big\}\\
&\leq  \mathlarger{\ex \, 1}_{I_\e}\Big\{C_p\Big(\delta\log\frac{T}{\delta}\Big)^{p/2}\Big\}\\
&= 1_{I_\e}\Big\{C_p\Big(\delta\log\frac{T}{\delta}\Big)^{p/2}\Big\}.
\end{align*}
In other words,
\begin{align}
&\lim_{N \to \infty} \prob(\sup_{n > N} \omega'(\tilde{\pi}^{(n)}, T, \delta) > \e)
\leq \mathlarger{1}_{I_\e}\Big\{C_p\Big(\delta\log\frac{T}{\delta}\Big)^{p/2}\Big\},
\end{align}
which is 0 when $\delta$ satisfies 
$$\ds \delta\log\frac{T}{\delta} < \frac{\e^2}{4^{(p+1)/p}\,C_p^{2/p}}.$$
With this chosen value of $\delta,$ take $N$ large enough so that

$$\prob\big(\sup_{n > N} \omega'(\tilde{\pi}^{(n)}, T, \delta) > \e\big) < \eta/2,$$
\noindent
then appropriately shrink $\delta$ until
\[
\ds \sum_{i=1}^N\prob(\omega'(\tilde{\pi}^{(i)}, T, \delta) > \e)< \eta/2,
\]
to conclude that
\[
\prob\big(\sup_n \omega'(\tilde{\pi}^{(n)}, T, \delta) > \e\big) < \eta.
\]
\end{proof}

\begin{proof}[Proof of Lemma \ref{lemma:expMartMoments}]
Since $V$ is continuous, it is bounded, and so it follows from Novikov's condition 
that $Z$ is a martingale. In fact, if $M(t)$ is a continuous local martingale,
$Z' := \exp(M - \frac{1}{2}\langle M\rangle)$ is a local martingale from It\^o's
lemma. Because it is
non-negative we may apply Fatou's lemma to an exhaustive sequence of local times
$T_n \overset{a.s.}{\to} \infty$ to see
$\ds \ex(Z'(t) | \mc{F}_s) \leq \lim_{n\to \infty}\ex(Z'(t\land T_n) |
\mc{F}_s) = \lim_{n\to \infty}Z'(s\land T_n) = Z'(s).$
That is, $Z'$ is a supermartingale. Take any $p, q, q' > 0$ with
$\frac{1}{q} + \frac{1}{q'} = 1.$ Then
\begin{align*}
\ex [Z(t)^p] &= \ex\Big[\exp\Big(-p\int_0^tV_s\, \mathrm{d}B_S - \frac{qp^2}{2}\int_0^tV_s^2\, \mathrm{d}s\Big)
\exp\Big(\frac{p(qp-1)}{2}\int_0^tV_s^2\, \mathrm{d}s\Big)\Big].
\end{align*}
Now apply Holder's inequality with $q, q'$
\begin{align*}
\ex [ Z(t)^p] &\leq \ex\Big[\exp\Big(-pq\int_0^tV_s\, \mathrm{d}B_s - \frac{q^2p^2}{2}\int_0^tV_s^2\, \mathrm{d}s\Big)\Big]
^{1/q}
\\
&\times \ex\Big[\Big(\frac{pq'(qp-1)}{2}\int_0^tV_s^2\, \mathrm{d}s\Big)\Big]^{1/q'}\\
&\leq 1 \cdot \ex\Big[\exp\Big(\frac{pq'(qp-1)}{2}\int_0^tV_s^2\, \mathrm{d}s\Big)\Big]^{1/q'}\\
&= \exp\Big(\frac{pq'(qp-1)}{2}\int_0^tV_s^2\, \mathrm{d}s\Big) < \infty.
\end{align*}
Here
 \[
 \ds \ex\Big[\exp\Big(-pq\int_0^tV_s\, \mathrm{d}B_s - \frac{q^2p^2}{2}\int_0^tV_s^2\, \mathrm{d}s\Big)\Big] \leq 1,
 \]
 since 
 \[
 \ds M(t) = -pq\int_0^tV_s\, \mathrm{d}B_s, \ \langle M\rangle(t) = q^2p^2\int_0^tV_s^2\, \mathrm{d}s
 \]
 and because $\exp (M(t) - \frac{1}{2}\langle M \rangle(t))$ is a supermartingale as explained above.
\end{proof}

\section{Uniqueness of the heat equation with free-boundary}
In this section we give existence and uniqueness 
for the PDE with free boundary condition $(p(t, \cdot), y(t))$ 
which
is the solution of our hydrodynamic limit given by \eqref{eq:pdeOfLimit}. If $(p, y)$ is a 
solution and $p(t, \cdot)$ represents the
distribution of heat, then the equation in Theorem \ref{HL} is interpreted as saying the 
acceleration of the moving barrier $y(t)$ is proportional to its temperature. The hydrodynamic limit
already yields existence of such a solution. In that statement of Theorem \ref{HL}
 $(\pi^{(n)}, Y^{(n)})$ converges in some sense to a solution of \eqref{eq:pdeOfLimit}. Here we show this is
the only solution by demonstrating uniqueness of this PDE with free boundary. 
\begin{remark}
For any solution $(p,\, y)$ of \eqref{eq:pdeOfLimit} make a substitution
$u(t, x) = p(t, x + y(t))$ and see $(u, \, y)$ is a classical solution to
\begin{align}
\begin{split}\label{eq:heatEqnTimeFree1}
&\frac{\partial u(t, x)}{\partial t} = \frac{1}{2}\frac{\partial^2 u(t, x)}{\partial x^2} + y'(t)\frac{\partial u(t, x)}{\partial x}, \text{ when } x > 0 ,
\\
&\frac{\partial^+ u(t, x)}{\partial x^+} = -2y'(t)u(t, x), \text{ at $x = 0$}\\
&y''(t) = -\frac{K}{2}u(t, 0), \ y(0) = 0, \, y'(0) = v \in \R, \
y'' \in C([0, T], \R),\\
&\lim_{t \downarrow 0} u(t, x) = f(x)dx, \ \, f \in L^1(\R_+).
\end{split}
\end{align}
In this way the two problems are equivalent.
\end{remark}
\begin{theorem}
The PDE problem \eqref{eq:heatEqnTimeFree1},
and equivalently that in \eqref{eq:pdeOfLimit}, has a unique solution for any 
$K\geq 0.$
\end{theorem}
\begin{remark}
The regularity of the boundary plays an important role because if $y''$ exists then
the solution to \eqref{eq:pdeOfLimit} has a stochastic representation given from
Corollary \ref{cor:densityOfRBM}.
We exploit this to show uniqueness.
\end{remark}
\begin{proof}
Theorem \ref{HL} gives existence. To show uniqueness we will prove
the corresponding barriers $y_1, y_2$ of any two solutions are in fact equal. 
Assume that $(p_1(t, \cdot), y_1(t)), \ (p_2(t, \cdot), y_2(t))$ are pairs solving 
the PDE with the given initial conditions.
Following Corollary \ref{cor:densityOfRBM} above we know that the 
transition density
$p_i(t, x)$ of Brownian motion reflecting from $y_i$ satisfies the PDE
\begin{flalign}\label{pde:limit_IC}
\begin{split}
&\frac{\partial p_i}{\partial t} = \frac{1}{2}\Delta_yp_i, \ y > y_i(t),\\
&\frac{\partial^+ p_i}{\partial y^+} = -2y_i'(t)p_i, \ y = y_i(t),\\
&\lim_{t \downarrow 0}p_i(t, y)\md y = f(y) \md y \in L^p(\R_+), p \geq 1
\end{split}
\end{flalign}
\begin{remark}\label{remark:initialCondition}
Given a $C^2$ function $y$ solving classical solution to the above PDE, one sees from Propositions \ref{densitysolution} and \ref{densityBounds} that $p(t, x)$ is density of 
Brownian motion reflecting from $y$ with initial condition $\xi \dist f(\md x)$. (For fixed $y$
there is a unique such $p(t, x)$, see Friedman \cite[Th. 15, Ch. 2]{Friedman}.) Consequently,
the second condition in \eqref{pde:limit_IC} is interpreted as convergence of the density of Brownian motion reflecting above $y$ to its value at time zero. From an analysis perspective, we have $\mc{W}_p(f(x)\md x, p(t, x)\md x) \to 0$, as $t \to 0$, which can easily be gathered from Propositions \ref{densitysolution} and \ref{densityBounds}. In particular, once we show this solution is unique, the stochastic representation shows $\int_{y(t)}^\infty p(t, x)\, \md x = 1$ and $\int_{y(t)}^\infty x^pp(t, x)\, \md x < \infty$ for all $t \in [0, T].$
\end{remark}
Without loss of generality we assume $\int f(y)\, \mathrm{d}y = 1.$ Let $(\Omega, \mathcal{F}, (\mathcal{F}_t)_{t \ge 0}, \prob,)$ be a probability space supporting an $\mathcal{F}_t-$adapted Brownian motion $B(t)$ and an independent 
random variable $\xi$ with distribution determined by $f(y)\md y$. As in the proof of Theorem
\ref{HL}, or Proposition \ref{prop:subsequentialLimit}, we know 
\[
\ds y_i'(t) = v -\frac{K}{2}\ex \, m_i(t), \ \text{ where } \ \ds m_i(t) = \max_{u \in [0, t]}\big(B(u) + \xi - y_i(u)\big)^-.
\] 
Linearity of expectation yields the following comparison between $y_1', y_2':$

\begin{align*}
|y_1(t) - y_2(t)| &\leq \int\limits_0^t|y_1'(s) - y_2'(s)|\, \mathrm{d}s\\
&= \frac{K}{2}\int\limits_0^t |\ex \, (m_1(s) - m_2(s))|\, \mathrm{d}s\\
&\leq \frac{K}{2}\int\limits_0^t\|y_1 - y_2\|_{[0, s]}\, \mathrm{d}s 
\leq \frac{K}{2}t\|y_1 - y_2\|_{[0, t]}.
\end{align*}

Because the right hand is nondecreasing this inequality holds when the left hand
is maximized across time, so
\begin{align*}
\|y_1 - y_2\|_{[0, t]} \leq \frac{K}{2}t\|y_1 - y_2\|_{[0, t]}.
\end{align*}
Therefore $\|y_1 - y_2\|_{[0, t]} \leq C\|y_1 - y_2\|_{[0, t]}$
for some $C < 1$ as long as $0 \leq t < t^* < 2/K.$ As a result
$\ds \|y_1 - y_2\|_{[0, t^*]} \, = 0$ for all $t^* \in [0, \sqrt{2/K}].$ In
other words, the barriers $y_1$ and $y_2$ are identical up until this fixed 
positive time. Repeating this argument shows that $y_1$ and $y_2$ are 
identical across the entire interval $[0, T].$ That is, there is a unique boundary $y$ associated to any pair $(p, y)$ solving \eqref{eq:pdeOfLimit}. Fixing $y,$ uniqueness of $p$ follows as a special case of the uniqueness result in Friedman \cite[Th. 15, Ch. 2]{Friedman} (or \cite[Th. 2.3]{burdzy2004}).
\end{proof}

\bibliography{hydrolimit_Revision3_imsart.bbl}
\bibliographystyle{plain}

\end{document}